\documentclass[11pt,a4paper]{amsart} 

\usepackage{graphicx}
\usepackage{enumerate}
\usepackage{amsmath, amssymb, amsthm, amscd}
\usepackage{dsfont}
\usepackage{url}
\usepackage{pst-all}

\theoremstyle{plain}
\newtheorem{theorem}{Theorem}[section]

\newtheorem{proposition}[theorem]{Proposition}

\theoremstyle{definition}
\newtheorem{definition}[theorem]{Definition}

\theoremstyle{remark}
\newtheorem{remark}[theorem]{Remark}

\theoremstyle{hypothesis}
\newtheorem{hypothesis}[theorem]{Hypothesis}

\numberwithin{equation}{section}

\newcommand{\ba}{\underline{a}} 
\newcommand{\br}{\underline{\rho}}
\newcommand{\bu}{\underline{u}}

\newcommand{\const}{\mathrm{const.}}
\newcommand{\circu}{\mathrm{circ}}

\newcommand{\abs}[1]{\lvert#1\rvert}
\newcommand{\D}{\partial}

\newcommand{\dive}{\mathrm{div}}
\newcommand{\diag}{\mathrm{diag}}

\newcommand{\Dlt}{\Delta t}

\newcommand{\Dt}{\partial_t}
\newcommand{\dvg}{\nabla \cdot }

\newcommand{\grd}{\nabla}

\newcommand{\ld}{\lambda}

\newcommand{\mbb}{\mathbb}
\newcommand{\mds}{\mathds}

\newcommand{\mcal}{\mathcal}

\newcommand{\norm}[1]{\lVert#1\rVert}

\newcommand{\solspc}{L^2(\mbb{T}^d)^{1+d}}

\newcommand{\veps}{\varepsilon}
\newcommand{\vrho}{\varrho}

\newcommand{\taud}{\mathbb{T}^d}
\newcommand{\Hdiv}{H(\dive;\mbb{T}^d)}
\newcommand{\Sint}{L^2(\mbb{T}^d)}

\makeindex             

\begin{document}

\title[An AP Low Mach Number Accurate Scheme for the Wave Equation
system]{Analysis of an Asymptotic Preserving Low Mach Number Accurate
  IMEX-RK Scheme for the Wave Equation System}  

\author[Arun]{K. R. Arun}
\address{School of Mathematics, Indian Institute of Science Education and Research Thiruvananthapuram, Thiruvananthapuram - 695551, India}
\email{arun@iisertvm.ac.in}

\author[Das Gupta]{A. J. Das Gupta}
\address{Department of Mathematics, Ramakrishna Mission Vidyamandira,
  Howrah - 711202, India}
\email{arnab.math@vidyamandira.ac.in}

\author[Samantaray]{S. Samantaray}
\address{School of Mathematics, Indian Institute of Science Education and Research Thiruvananthapuram, Thiruvananthapuram - 695551, India}
\email{sauravsam13@iisertvm.ac.in}

\date{\today}

\subjclass[2010]{Primary 35L45, 35L60, 35L65, 35L67, 35L05; Secondary 65M06,
  65M08}

\keywords{Compressible Euler system, Incompressible Euler system, The
  wave equation system, Zero Mach number limit, IMEX-RK schemes,
  Asymptotic preserving, Asymptotic accuracy, Finite volume method}

\begin{abstract}
  In this paper the analysis of an asymptotic preserving (AP) IMEX-RK
  finite volume scheme for the wave equation system in the zero Mach
  number limit is presented. The accuracy of a numerical scheme at low
  Mach numbers is its ability to maintain the solution close to the
  incompressible solution for all times, and this can be formulated in
  terms of the invariance of a space of constant densities and
  divergence-free velocities. An IMEX-RK methodology is employed to
  obtain a time semi-discrete scheme, and a space-time fully-discrete
  scheme is derived by using standard finite volume techniques. The
  existence of a unique numerical solution, its uniform stability with
  respect to the Mach number, the AP property, and the accuracy at low
  Mach numbers are established for both time semi-discrete, and
  space-time fully-discrete schemes. Extensive numerical case studies
  confirm uniform second order convergence of the scheme with respect
  to the Mach number, and all the above-mentioned properties.
\end{abstract}

\maketitle

\section{Introduction}
\label{sec:introduction}

Singular perturbation problems containing small parameters arise in
the mathematical modelling of several problems in science and
engineering. A typical example of a singularly perturbed problem is
the wellknown low Mach number flow in fluid dynamics, often
encountered in magnetohydrodynamics, atmospheric, and geophysical
flows, weather modelling, combustion theory, and so on. The Euler
equations of motion provide a simple yet optimal mathematical
tool to model and simulate many of the aforementioned physical
processes. The scaled, isentropic Euler equations read 
\begin{equation}
  \label{eq:isentropic_euler_scaled}
  \begin{aligned}
    \Dt \rho + \dvg (\rho u) &= 0, \\ 
    \Dt (\rho u) + \dvg (\rho u \otimes u ) + \frac{\nabla p}{\veps^2}
    &= 0, 
\end{aligned}  
\end{equation} 
where the independent variables are time $t > 0$, and space $x \in 
\mbb{R}^d$, $d = 1,2,3$, and the dependent variables are $\rho =
\rho(t,x)>0$, the density, and $u = u(t,x) \in \mbb{R}^d$, the velocity
of the fluid.The pressure $p$ is given by the equation of state
$p=P(\rho):=\rho^\gamma$, where $\gamma$ a constant. Here, $\veps$ is the
ratio of a reference fluid velocity to a reference sound velocity, and is
known as the reference Mach number. In low Mach number flows,
i.e.\ when $\veps \sim 0$, $\veps$ plays the role of a singular
perturbation parameter, and it is wellknown that in the limit
$\veps\to0$, the solutions of \eqref{eq:isentropic_euler_scaled}
approximate their incompressible counterparts; see
\cite{klainerman-majda} for more details. On the other hand, numerical
schemes designed  for the system \eqref{eq:isentropic_euler_scaled}
suffer from a lot of predicaments in the limit $\veps\to0$; see, e.g.\
\cite{klein}, and the references therein. From a numerical analysis
point of view, the main challenges faced by numerical schemes in a low
Mach number regime are the stiffness arising due to stringent CFL
restrictions, creation of spurious waves, dependence of the numerical
viscosity of a scheme on the Mach number leading to lack of
stability, and the inability to respect the transitional behaviour of
the system of equations in the singular limit.

In our previous work \cite{ArunSamArxiv}, a second order accurate,
semi-implicit finite volume approximation for the Euler system
\eqref{eq:isentropic_euler_scaled} in a low Mach number regime was
proposed and implemented. It is shown that the above mentioned scheme 
overcomes the severe CFL restrictions, avoids the generation of
spurious waves, and is consistent in the singular limit. This paper is
aimed to present a rigorous analysis of the semi-implicit scheme of
\cite{ArunSamArxiv} by considering the linear wave equation system as
a simplified model. The  scaled, purely hyperbolic, linear wave
equation system with advection is given by  
\begin{equation}
  \label{eq:LWESWA}
  \begin{aligned}
    \Dt \vrho + (\bu \cdot \nabla) \vrho + \frac{\ba}{\veps} \dvg
    u &= 0, \\
    \Dt u + (\bu \cdot \nabla) u + \frac{\ba}{\veps} \nabla \vrho &
    = 0. 
  \end{aligned}
\end{equation}
Here, $\vrho$ denotes a scaled density with $\rho(t,x) = \br \left(
  1 + \frac{\veps}{\ba}\vrho(t,x) \right)$, and the constants $\br,
\bu$, and $\ba$ are, respectively, the linearisation states of $\rho,
u$ and the sound velocity $a$.

Expanding all the dependent variables in \eqref{eq:LWESWA} using the
ansatz 
\begin{equation}
  \label{eq:f_ansatz}
  f(t,x)=f_{(0)}(t,x)+\veps f_{(1)}(t,x)+\veps^2 f_{(2)}(t,x) + \ldots ,
\end{equation}
and performing a scale analysis with the use of appropriate boundary
conditions yields the following linearised, mixed hyperbolic-elliptic,
incompressible system:
\begin{equation}\label{eq:LWESWA_LM}
  \begin{aligned}
    \vrho_{(0)} &= \const, \\
    \Dt u_{(0)} + (\bu \cdot \nabla) u_{(0)} + \ba \nabla \vrho_{(1)} &=
    0, \\
    \dvg u_{(0)} &= 0
\end{aligned}
\end{equation}
for the unknowns $u_{(0)}$ and $\vrho_{(1)}$. 

It is essential for a discretisation of \eqref{eq:LWESWA} to respect
the limiting system \eqref{eq:LWESWA_LM} in the asymptotic limit
$\veps\to0$. In addition, it is desirable that its stability
restrictions are independent of $\veps$ so that the stiffness arising
when $\veps\to0$ can be overcome. These two requirements fall under
the purview of the so-called asymptotic preserving (AP) methodology
\cite{Degond13,jin-ap, Jin12}. In a related work
\cite{degond_negulescu_etal}, the authors study a strongly anisotropic
singular elliptic problem arising in plasma physics. In the singular limit,
i.e.\ when the anisotropy parameter goes to zero, the problem converges
to an ill-posed one. As a cure, the authors propose an AP
reformulation of the given anisotropic equation in order to write the
original problem in such a way that a continuous transition towards
the limit can be achieved. Hence, the AP methodology can also help to
get rid of the possible ill-posedness which might arise in the
singular limit. 

It is well known that being AP is predominantly dictated by the
particular time discretisation chosen. The scheme designed in
\cite{ArunSamArxiv} makes use of the implicit explicit Runge-Kutta
(IMEX-RK) time discretisation and an appropriate flux decomposition to 
achieve the AP property for the Euler system
\eqref{eq:isentropic_euler_scaled} in the zero Mach number limit; see
\cite{AscherRuuthSpiteri, BispArunLucaNoe, cordier, degond_tang,
  NoelleArunMunz,  pareschi-russo-conf, pareschi-russo,
  zakerzadeh-noelle} for more detailed discussions on IMEX-RK schemes,
AP schemes, and their other applications. As mentioned in the
beginning, compressible flow solvers suffer from severe loss of
accuracy at low Mach numbers due to the creation of spurious
waves. Recently, a detailed analysis was carried out by Dellacherie in
\cite{dellacherie} on the behaviour of explicit Godunov-type schemes
in a low Mach number regime. The above mentioned study using the
linear wave equation system \eqref{eq:LWESWA} reveals that the
inaccuracies can be avoided by enforcing the particular scheme to
preserve a space of constant densities and divergence-free velocities,
known as the well-prepared space. Following this, in
\cite{ArunSamArxiv}, a scheme which leaves the well-prepared space
invariant is designated as asymptotically accurate (AA).   

The goal of the current paper is to present a detailed analysis of a
time semi-discrete as well as space-time fully-discrete IMEX-RK scheme
for the wave equation system \eqref{eq:LWESWA} under the low Mach
number scaling. The study focusses on addressing the following three
issues:
\begin{enumerate}
\item\label{item1:solvability}
the existence of a unique numerical solution for a fixed $\veps >0$;
\item\label{item2:AP}
the uniform stability of the numerical solution with respect to
$\veps$, and the AP property;
\item\label{item3:E-inv}
the invariance of the well-prepared space by the numerical solution
implying asymptotic accuracy.
\end{enumerate}

The time semi-discrete scheme corresponds to the dual formulation of
an elliptic equation for the density, and the existence and uniqueness of
its solution is obtained via the saddle point theory of variational
problems. The asymptotic consistency, as done, e.g.\ in
\cite{ArunSamArxiv}, then reveals that we do not encounter
pathologies, such as ill-posedness of the limit problem, cf.\
\cite{degond_negulescu_etal}. The fully-discrete scheme obtained by
simple central differencing involves circulant matrices
\cite{gray}. We exploit the theory of these matrices to establish the
above properties for the fully-discrete setup.  

In order to carry out the analysis, the linear wave equation system
\eqref{eq:LWESWA} is rewritten in the evolution form      
\begin{equation}
\Dt U + H(U) + \frac{1}{\veps} L(U) = 0, \label{eq:LWESWA_ev} 
\end{equation}
 via the operators $H$ and $L$ defined as
\begin{equation}
  \label{eq:Notn1}
  U := \begin{pmatrix}
    \vrho\\
    u
  \end{pmatrix}, \
  H(U) := \begin{pmatrix}
    \bu \cdot \grd \vrho \\
    (\bu \cdot \grd) u 
  \end{pmatrix}, \ \mbox{and} \
  L(U) := \begin{pmatrix}
    \ba \dvg u \\
    \ba \grd \vrho
  \end{pmatrix}.	
\end{equation} 
Here, $H$ is the convective operator with a timescale of order $1$
and $L/\veps$ is the acoustic operator with a timescale of the order
of $\veps$. As discussed above, when $\veps\to0$, the solutions
$U=(\vrho,u)$ of \eqref{eq:Notn1} converge to
$U_{(0)}=(\vrho_{(0)},u_{(0)})$ in the well-prepared space, which
happens to be the kernel of the operator $L$. Hence, as done in
\cite{dellacherie}, throughout this paper, we restrict our analysis to
the following initial value problem:   
\begin{equation}
  \label{eq:ivpALWES}
  \begin{aligned}
    \Dt U + \frac{1}{\veps} L(U) &= 0, \\
    U(0,x)&=U_0(x), \ x\in\mbb{T}^d,
  \end{aligned}
\end{equation}
where $\mbb{T}^d$ denotes the $d$-dimensional torus to take into
account of the periodic boundary conditions. However, the numerical case
studies are performed also on the model \eqref{eq:LWESWA_ev} with
advection. 

The rest of this paper is organised in the following way. In
Section~\ref{sec:wave-equation} we briefly recall the results from
\cite{dellacherie}, which are relevant for the present
study. Section~\ref{sec:AP_AA_IMEX-RK} is devoted to a short
presentation of IMEX-RK time discretisation for stiff systems of ODEs, 
and the notions of AP and AA properties. The analysis of the time
semi-discrete scheme obtained after employing the IMEX-RK method is
taken up in Section~\ref{sec:TimeSemiDiscrete}, where we prove the
desired properties mentioned above in 
(\ref{item1:solvability})-(\ref{item3:E-inv}). In
Section~\ref{sec:SpaceTimeFullyDiscrete}  we present a space-time
fully-discrete scheme derived by using a finite volume technique. The
theory of circulant matrices is used to establish the same 
properties (\ref{item1:solvability})-(\ref{item3:E-inv}) for the
fully-discrete scheme. The results of numerical case studies are
reported in Section~\ref{sec:Numerics}, where we numerically
corroborate the theoretical claims. Finally, the paper is
concluded with some remarks in Section~\ref{sec:Conclusion}.
 
\section{Analysis of the Wave Equation System}  
\label{sec:wave-equation}

In this section, we briefly recall some of the results from
\cite{dellacherie}, regarding the low Mach number limit of the wave
equation system \eqref{eq:LWESWA_ev}. First, we consider the space of
solutions of \eqref{eq:LWESWA_ev} which is the Hilbert space of square
integrable functions $\solspc$. The space $\solspc$ is equipped with
the innerproduct 
\begin{equation}\label{eq:U1U2ip}
 (U_1,U_2):= (\vrho_{1},\vrho_{2})+(u_{1}, u_{2}),
\end{equation}
where $U_m=(\vrho_m,u_m),m=1,2$, in the above, and throughout the rest 
of this paper, $(\cdot,\cdot)$ denotes the $L^2$ innerproduct. The
kernel of the wave operator $L$ is given by   
\begin{equation}\label{eq:LWES_E_def}
  {\mcal{E}}:= \left\{ U \in \solspc \colon \grd \vrho = 0, \
               \text{and}  \ \dvg u = 0 \right\}, 
\end{equation}
which is the so-called well-prepared, incompressible, space of constant
densities and divergence-free velocities. The orthogonal complement
$\tilde{\mcal{E}}$ of $\mcal{E}$ is defined as
\begin{equation}\label{eq:LWES_E_comp_def}
\tilde{\mcal{E}} := \left\{ U \in \solspc \colon \int_{\mbb{T}^d} \rho dx
                     = 0,  \ \text{and}  \ \nabla \times u = 0\right\}.
\end{equation}
The spaces $\mcal{E}$ and $\tilde{\mcal{E}}$, given by
\eqref{eq:LWES_E_def} and \eqref{eq:LWES_E_comp_def}, yields the
following Helmholtz-Hodge-Leray decomposition of $\solspc$:
\begin{equation}
  \label{eq:LWES_HelmLer}
  \mcal{E} \oplus \tilde{\mcal{E}} = \solspc,   
  \ \mbox{and } \  \mcal{E} \perp \tilde{\mcal{E}}.
\end{equation}
As a consequence of \eqref{eq:LWES_HelmLer}, we can decompose any $U
\in \solspc$; there exists a unique $\hat{U} \in \mcal{E}$ and
$\tilde{U}\in \tilde{\mcal{E}}$, such that $U = \hat{U} +
\tilde{U}$. We define the Helmholtz-Hodge-Leray projection
$\mbb{P}\colon \solspc\to \mcal{E}$, via    
\begin{equation}
  \label{eq:LWES_Pdefn}
  \mbb{P}U := \hat{U}.
\end{equation}
\begin{definition}\label{LWES_energy_defintion}
  The energy $E$ of the system \eqref{eq:LWESWA}  is defined as 
  \begin{equation}\label{eq:LWES_energy_def}
    E :=\frac{1}{2} (U, U).
  \end{equation}
\end{definition} 
\begin{proposition}\label{energyconserv}
Let U be a solution of the system \eqref{eq:LWESWA} on
$\mbb{T}^d$. Then, the energy $E$ introduced in
Definition~\ref{LWES_energy_defintion} is preserved, i.e.\
\begin{equation}\label{eq:LWES_ener_cons}
E(t) = E(0),  \ \text{for all} \  t >0.
\end{equation}
\end{proposition} 
\begin{remark}
Proposition \ref{energyconserv} states that the linear wave equation
system \eqref{eq:LWESWA} conserves the energy
\eqref{LWES_energy_defintion}. 
\end{remark}
\begin{proposition}\label{E_Ep_inv}
Let $U$ be the solution of the system \eqref{eq:LWESWA} with initial
data $U_{0}$. Then, 
\begin{itemize}
\item for all $U_{0} \in \mcal{E}$, we have $U(t, \cdot) \in \mcal{E}$
  for all $t >0$;  
\item for all $U_{0} \in \tilde{\mcal{E}}$, we have $U(t, \cdot) \in
  \tilde{\mcal{E}}$ for all $t >0$.    
\end{itemize}
\end{proposition}
\begin{remark}
Proposition \ref{E_Ep_inv} states that the wave equation system
leaves both the spaces $\mcal{E}$ and $\tilde{\mcal{E}}$ invariant. In
other words, if the solution lives in one of these spaces initially,
then it lives there for all times.   
\end{remark}
\begin{theorem}[\cite{dellacherie}] 
  \label{propLin}
  Let $U$ be a solution of the IVP: 
  \begin{align} 
    \Dt U + H(U) + \frac{1}{\veps} L(U) & = 0, \ t>0, \
                                            x\in\mbb{T}^d, \label{eq:PropEqHL}\\
    U(0,x) &= U_0(x),	\ x\in\mbb{T}^d, \label{eq:PropIcHL}
  \end{align}
  and let $\bar{U}$ be a solution of the IVP: 
  \begin{align}
    \Dt \bar{U} + H(\bar{U}) &= 0, \ t>0, \ x\in\mbb{T}^d, \label{eq:PropEqH}\\
    \bar{U}(0,x) &= \hat{U}_0(x), \ x\in\mbb{T}^d, \label{eq:PropIcH} 
  \end{align}
  where $\hat{U}_0:=\mbb{P}U_0$. Let $U=\hat{U}+\tilde{U}$ be the
  Helmholtz-Hodge-Leray decomposition of $U$. Then, the following holds.
  \begin{enumerate}[(i)]
  \item \label{cond1} 
    $\hat{U} = \bar{U}$, 
  \item \label{cond2}
    $\tilde{U}$ is the solution of \eqref{eq:PropEqHL} with
    initial condition $\tilde{U}_0:=(\mbb{I}-\mbb{P})U_0$.
  \end{enumerate}
  Moreover, there holds the energy conservation:
  \begin{equation}
    \label{eq:PropEnCons}
    E_{in}(t) = E_{in}(0)  \ \text{and} \  E_{ac}(t) = E_{ac}(0), \
    \text{for all} \ t>0,
  \end{equation}
  where $E_{in}:=(\hat{U},\hat{U})$ and
  $E_{ac}:=(\tilde{U},\tilde{U})$. As a consequence, the following
  estimate holds: 
  \begin{equation}
    \norm{U_0-\mbb{P}U_0} = \mcal{O}(\veps)
   \implies
   \norm{U(t)-\mbb{P}U(t)} = \mcal{O}(\veps) \ \mbox{for all} \ t>0. 
    \label{eq:estSchochetLin}
  \end{equation}
\end{theorem}

Proposition~\ref{propLin} lies at the core of the analysis of
numerical schemes presented in \cite{dellacherie}. Depending on the
order of accuracy, numerical schemes introduce numerical diffusion,
dispersion or higher order correction terms in the modified partial
differential equations (MPDE). However, we desire that the numerical
solutions which satisfy the MPDE also exhibit properties close to
those of the solutions of the continuous system. One of the key
properties is to satisfy the estimate \eqref{eq:estSchochetLin} which
states that a solution remains close to $\mcal{E}$ for all times $t >
0$ if it is so at time $t=0$. The following proposition guarantees a
sufficient condition to ensure \eqref{eq:estSchochetLin}, which also
accommodates any general linear discretisation. 
\begin{proposition}\cite{dellacherie}
\label{F_einv}
Let $U$ be a solution of the IVP:
\begin{align} 
    \Dt U + \mcal{F}_x U &= 0, \ t>0, \
    x\in\mbb{T}^d, \label{eq:mThmEq}  \\
    U(0,x) &= U_0(x), \ x\in\mbb{T}^d, \label{eq:mThmIc} 
  \end{align}
  which is assumed to be well-posed in
  $L^\infty\left([0,\infty);\solspc\right)$, with $\mcal{F}_x $ a
  linear spatial differential operator. Then the following conclusions
  hold.
  \begin{enumerate}[(i)]
  \item \label{ThmDelCond1}
    The solution $U$ satisfies the estimate 
    \begin{equation} \label{eq:ThmDelEst1}
      \norm{U_0-\mbb{P}U_0} = \mcal O(\veps) \implies
      \norm{U(t)-\bar{U}(t)} = \mcal O(\veps), \ \text{for all} \ t>0,  
    \end{equation}
    where $\bar{U}$ is a solution of \eqref{eq:mThmEq} with the initial
    condition $\bar{U}(0) = \mbb{P}U_0$. However, we don't have the 
    apriori estimate $\norm{U(t) - \mbb{P}U(t)} = \mcal O(\veps)$ for
    all $t>0$.
  \item \label{ThmDelCond2}
    When the operator $\mcal{F}_x$ leaves $\mcal{E}$ invariant, i.e.\
    whenever $U_0\in\mcal{E}$ implies $U(t)\in\mcal{E}$ for all $t>0$,
    then $U$ satisfies the estimate \eqref{eq:ThmDelEst1}, and in
    addition we have 
    \begin{equation} 
      \label{eq:ThmDelEst2}
      \norm{U_0-\mbb{P}U_0} = \mcal O(\veps) \implies \norm{U(t) -
        \mbb{P}U(t)} = \mcal O(\veps)  \ \text{for all} \ t>0.
    \end{equation}
  \end{enumerate}
\end{proposition}
\begin{remark}
  The second  part of the above Proposition~ \ref{F_einv} give us the
  importance of $\mcal{E}$ invariance for any numerical scheme. It states
  that if a scheme is $\mcal{E}$-invariant, then if the initial data
  $U_{0}$ is almost in the well-prepared subspace $\mcal{E}$, the
  solution at all later times also lives close $\mcal{E}$. Loosely
  speaking, the estimate \eqref{eq:ThmDelEst2} states that if the
  initial data is almost incompressible, then the solution for all time
  $t>0$ is also almost incompressible. It was observed in
  \cite{dellacherie} that satisfying the condition
  \eqref{eq:ThmDelEst2} avoids the creation of spurious acoustic waves
  in the numerical solution. Our numerical experiments reported in
  Section~\ref{sec:Numerics} clearly validate this observation. 
\end{remark}
\section{Asymptotic Preserving and Asymptotically Accurate IMEX-RK  
  Schemes}
\label{sec:AP_AA_IMEX-RK}

We devote this section to recall the notions of AP and AA schemes as
done in \cite{ArunSamArxiv}. Further, in a nutshell, we review the
wellknown IMEX-RK schemes for stiff systems of ODEs which are employed to
approximate the time derivatives in system \eqref{eq:LWESWA}.

\subsection{Asymptotic Preserving Property}
One of the essential properties of a numerical approximation for a
singular perturbation problem is its ability to capture the solution of
the limit system as well as the solution of the original problem. The
AP methodology not only provides a framework to address the
convergence of the numerical solution to that of the limit system but
also takes care of the stability restrictions such that they don't 
deteriorate in the singular limit.
\label{sec:ap-schemes}
\begin{definition}
\label{defn:AP}
  Let $\mcal{P}^{\veps}$ denote a singularly perturbed problem with the
  perturbation parameter $\veps$. Let $\mcal{P}^0$ denote the limiting
  system of $\mcal{P}^\veps$ for $\veps \to 0$. A discretisation
  $\mcal{P}_{h}^{\veps}$ of  $\mcal{P^{\veps}}$, with $h$ being the
  discretisation parameter, is called AP, if 
  \begin{enumerate}[(i)]
    \item $\mcal{P}^{0}_h$ is a consistent discretisation of the
      problem $\mcal{P}^0$, called the asymptotic consistency, and
    \item the stability constraints on $h$ are independent of $\veps$,
      called the asymptotic stability.
  \end{enumerate} 
In other words, the following diagram commutes:
\begin{equation*}
    \begin{CD}
      \mcal{P}^{\veps}_{h} @>{h\to 0}>> \mcal{P}^{\veps}\\
      @VV{\veps\to 0}V  @VV{\veps\to 0}V \\
      \mcal{P}^{0}_{h} @>{h\to 0}>> \mcal{P}^0
    \end{CD} 
  \end{equation*}
\end{definition}
\subsection{Asymptotic Accuracy}
\label{sec:aa-schemes}

We note from Proposition~\ref{E_Ep_inv} that a solution of the wave
equation system lives in $\mcal{E}$ at all times if the initial data
is taken from $\mcal{E}$. It was shown in \cite{dellacherie} that a
sufficient condition for a numerical scheme for the wave equation
system to be accurate at low Mach numbers, i.e.\ it is free from the
creation of spurious waves, is the $\mcal{E}$-invariance. Based on
this idea, in \cite{ArunSamArxiv}, the notion of asymptotic
accuracy is defined as the following.  
\begin{definition}
  \label{defn:AA}
  A numerical approximation for the wave equation system
  \eqref{eq:LWESWA} is said to be asymptotically accurate (AA), if it
  leaves the incompressible subspace $\mcal{E}$ invariant.   
\end{definition}

\subsection{IMEX-RK Time Discretisation}
IMEX-RK schemes provide a robust and efficient framework to design AP
schemes for singular perturbation problems. In this work, we only
consider a subclass of the IMEX-RK schemes, namely diagonally implicit
or (DIRK) schemes. An $s$-stage IMEX-RK scheme is characterised by the
two $s\times s$ lower triangular matrices $\tilde{A}=
(\tilde{a}_{i,j})$, and $A=(a_{i,j})$, the coefficients $\tilde{c}=(\tilde{c}_1,
\tilde{c}_2,\ldots,\tilde{c}_s)$ and $c=(c_1, c_2, \ldots, c_s)$, and
the weights $\tilde{\omega}=(\tilde{\omega}_1, \tilde{\omega}_2,
\ldots,\tilde{\omega}_s)$ and $\omega = (\omega_1, \omega_2, \ldots
,\omega_s)$. Here, the entries of $\tilde{A}$ and $A$ satisfy the
conditions $\tilde{a}_{i,j}=0$ for $ j \geq i$, and $a_{i,j}=0$ for
$j>i$. Let us consider the following stiff system of ODEs in an
additive form:   
\begin{equation}
  \label{eq:stiff_ODE}
  y^\prime = f(t,y) + \frac{1}{\veps} g(t,y),
\end{equation}
where $0<\veps\ll 1$ is called the stiffness parameter. The functions $f$ and
$g$ are known as, respectively, the non-stiff part and the stiff part of the
system \eqref{eq:stiff_ODE}; see, e.g.\ \cite{hairer-wanner-2}, for a
comprehensive treatment of such systems. 

Let $y^n$ be a numerical solution of \eqref{eq:stiff_ODE} at
time $t^n$ and let $\Dlt$ denote a fixed timestep. An $s$-stage
IMEX-RK scheme, cf.,\ e.g.\ \cite{AscherRuuthSpiteri,
pareschi-russo-conf}, updates $y^n$ to $y^{n+1}$ through $s$ 
intermediate stages:  
\begin{align}
  Y_i &= y^n + \Dlt \sum\limits_{j=1}^{i-1}\tilde{a}_{i,j} f(t^n + \tilde{c}_j\Dlt, Y_j) + 
        \Dlt\sum \limits_{j=1}^s a_{i,j} \frac{1}{\veps}g(t^n + c_j
        \Dlt,Y_j), \ 1 \leq i \leq s, \label{eq:imex_Yi} \\
  y^{n+1} &= y^n  + \Dlt \sum\limits_{i=1}^{s}\tilde{\omega}_{i} f(t^n
            + \tilde{c}_i\Dlt, Y_i) + \Dlt\sum \limits_{i=1}^s
            \omega_{i}\frac{1}{\veps} g(t^n +
            c_i\Dlt,Y_i). \label{eq:imex_yn+1} 
\end{align}



In order to further simplify the analysis of the schemes presented in
this paper, we restrict ourselves only to two types of DIRK schemes,
namely the type-A and type-CK schemes which are defined below; see
\cite{KENNEDY2003139} for details. 
\begin{definition}
\label{eq:typeA_CK}
An IMEX-RK scheme is said to be of 
\begin{itemize}
\item type-A, if the matrix $A$ is invertible; 
\item type-CK, if the matrix $A \in \mbb{R}^{s \times s}, \ s \geq 2$,
  can be written as  
\begin{equation*}  
  A = 
  \begin{pmatrix}
    0 & 0 \\
    \alpha & A_{s-1 }
  \end{pmatrix},
\end{equation*}
where $\alpha \in \mbb{R}^{s-1} $ and $A_{s-1} \in \mbb{R}^{s-1 \times
s-1}$ is invertible.
\end{itemize} 
\end{definition}
The results presented in later sections, are obtained using both the
first order Euler(1,1,1), and second order ARS(2,2,2) schemes for time
discretisations; see \cite{pareschi-russo-conf, pareschi-russo} for
their definitions. Here, in the triplet $(s,\sigma,p)$, $s$ is the number of
stages of the implicit part, the number $\sigma$ gives the number of
stages for the explicit part and $p$ gives the overall order of the
scheme. We refer the interested reader to \cite{hairer-wanner-2,
  KENNEDY2003139, pareschi-russo-conf, pareschi-russo} and the
references therein for a detailed study of IMEX-RK schemes. 
\begin{hypothesis}
  \label{hyp:H2}
  We suppose that the IMEX-RK scheme under consideration is of Type-A
  or Type CK. 
\end{hypothesis}

\section{Time Semi-discrete Scheme and Its Analysis}
\label{sec:TimeSemiDiscrete}

In this section we present our time semi-discrete scheme for the wave
equation system \eqref{eq:LWESWA} obtained after approximation of 
the time derivatives using the IMEX-RK methodology described in
Section~\ref{sec:AP_AA_IMEX-RK}. We carry out a detailed analysis of
the scheme, and show some if its key properties, namely its
solvability, i.e.\ the existence of a numerical solution for any fixed 
$\veps>0$, the AP property and the asymptotic accuracy.   

\subsection{Time Semi-discrete Scheme}
\label{sec:time_semi-disc}

As a first step in defining a time semi-discrete scheme, we split the
fluxes in \eqref{eq:LWESWA} into a stiff and a non-stiff part,
yielding  
\begin{equation}
  \label{eq:flux_split}
  G(U) := \frac{\ba}{\veps} \begin{pmatrix}
    u \\
    \vrho 
  \end{pmatrix}, \ 
  F(U) := \begin{pmatrix}
    \vrho \bu \\
    \bu \otimes u
  \end{pmatrix}. 
\end{equation}

Let $0<t^1<t^2<\cdots<t^n<\cdots$ be an increasing sequence of
times. In the following, $f^n$ denotes an approximation to the value
of a function $f(t,x)$ at time $t^n$, i.e.\ $f^n(x)\sim
f(t^n,x)$. Treating $F$ explicitly, and $G$ implicitly, the IMEX-RK
time semi-discrete scheme can be obtained as follows. 
\begin{definition}
  Given an approximation $(\vrho^n, u^n)$ of the numerical solution at 
  time $t^n$, and a timestep $\Dlt$, the $k^{th}$ stage of an
  $s$-stage IMEX-RK scheme for the wave equation system
  \eqref{eq:LWESWA} is defined by 
\begin{align}
  \vrho^{k} &= \vrho^n - \Dlt \tilde{a}_{k,\ell}(\bu\cdot\nabla)\vrho^{\ell}
               -\Dlt a_{k,l}\frac{\ba}{\veps}\dvg u^{l}, \ \mbox{for
              each} \ k=1,2,\ldots,s, \label{eq:weq_varrho_k_semi_disc} \\
  u^{k}&=u^n-\Dlt \tilde{a}_{k,\ell} (\bu\cdot\nabla)u^{\ell} 
           - \Dlt a_{k,l}\frac{\ba}{\veps}\nabla\vrho^{l},  \
         \mbox{for each} \ k=1,2,\ldots,s.
           \label{eq:weq_u_k_semi_disc}
\end{align}
The approximate numerical solutions $\vrho^{n+1}$ and $u^{n+1}$ at
time $t^{n+1}$ are defined as 
\begin{align}
  \vrho^{n+1} &= \vrho^n - \Dlt \tilde{\omega}_k (\bu\cdot\nabla)\vrho^k
               -\Dlt \omega_k \frac{\ba}{\veps}\dvg
               u^{k}, \label{eq:weq_varrho_semi_disc_finupd} \\ 
  u^{n+1}&=u^n-\Dlt \tilde{\omega}_k  (\bu\cdot\nabla)u^k 
           - \Dlt \omega_k \frac{\ba}{\veps}\nabla\vrho^{k}.
           \label{eq:weq_u_semi_disc_finupd} 
\end{align}
\end{definition}

In the above, and throughout the rest of this paper, we follow the
convention that a repeated index always denotes the summation with
respect to that index. Here, the index $k$ assumes values in
$\{1,2,\ldots,s\}$, and the indices $\ell$ and $l$ are used to denote, 
respectively, the summation in the explicit and implicit terms,
i.e.\ they assume values in the sets $\{1,2, \ldots, k-1\}$ and
$\{1,2,\ldots,k\}$.

\subsection{Solvability of the Time Semi-discrete Scheme}

The aim of this subsection is to establish the existence of a numerical
solution to
\eqref{eq:weq_varrho_k_semi_disc}-\eqref{eq:weq_u_semi_disc_finupd}
using variational formulations, and classical saddle point theory;
see, e.g.\ \cite{brezzi-fortin-book,ciarlet-book} for more details. To
this end, let us consider the standard function spaces    
\begin{equation}
  \label{eq:Hdiv_L2}
  V:=\Hdiv, \ \mbox{and} \ M:=\Sint. 
\end{equation}
In
\eqref{eq:weq_varrho_k_semi_disc}-\eqref{eq:weq_u_semi_disc_finupd},
we multiply the density updates by a test function $\ld\in M$ and the
velocity updates by a test function $v\in V$, and integrate by parts to
get the following weak formulation.

For $k=1,2,\ldots,s$, find $(u^k,\vrho^k)\in V\times M$ satisfying
\begin{align}
  (u^k,v)-\Dlt a_{k,k}\frac{\ba}{\veps}(\dvg
  v,\vrho^k)&=(u^n,v)+\Dlt a_{k,\ell}\frac{\ba}{\veps}(\dvg
              v,\vrho^\ell), \ \mbox{for all} \ v\in
              V,\label{eq:imp_wform_u_k} \\  
  -\Dlt a_{k,k}\frac{\ba}{\veps}(\dvg
  u^k,\ld)-(\vrho^k,\ld)&=-(\vrho^n,\ld)+\Dlt a_{k,\ell}\frac{\ba}{\veps}(\dvg
  u^\ell,\ld),  \ \mbox{for all} \ \ld\in
                          M. \label{eq:imp_wform_rho_k}  
\end{align}
Finally, find $(u^{n+1},\vrho^{n+1})\in V\times M$ satisfying
\begin{align}
  (u^{n+1},v)&=(u^n,v)+\Dlt \omega_k\frac{\ba}{\veps}(\dvg v,\vrho^k),
               \ \mbox{for all} \ v\in V,\label{eq:imp_wform_u_n+1} \\  
  (\vrho^{n+1},\ld)&=(\vrho^n,\ld)-\Dlt \omega_k\frac{\ba}{\veps}(\dvg
                     u^k,\ld),  \ \mbox{for all} \ \ld\in
                     M. \label{eq:imp_wform_rho_n+1} 
\end{align}
Note that the semi-discrete scheme
\eqref{eq:weq_varrho_k_semi_disc}-\eqref{eq:weq_u_semi_disc_finupd}
admits a solution if, and only if, the weak formulations
\eqref{eq:imp_wform_u_k}-\eqref{eq:imp_wform_rho_n+1} admit a
solution. 
\begin{theorem}
  \label{thm:time_disc_existence}
  Suppose an approximation $(u^n,\vrho^n)\in V\times M$ of the
  numerical solution at time $t^n$, and a timestep $\Dlt$ are chosen.
  Then, under the Hypothesis~\ref{hyp:H2}, the weak formulations 
  \eqref{eq:imp_wform_u_k}-\eqref{eq:imp_wform_rho_k} are uniquely
  solvable for a $(u^k,\vrho^k)\in V\times M$ for
  $k=1,2,\ldots,s$. Consequently,
  \eqref{eq:imp_wform_u_n+1}-\eqref{eq:imp_wform_rho_n+1} defines
  uniquely an approximate numerical solution $(u^{n+1},\vrho^{n+1})\in
  V\times M$ at time $t^{n+1}=t^n+\Dlt$.    
\end{theorem}

The weak formulations 
\eqref{eq:imp_wform_u_k}-\eqref{eq:imp_wform_rho_k} can be recast in
the saddle point form: find $(u^k,\vrho^k)\in V\times M$ satisfying 
\begin{equation}\label{eq:ak_bk}
\begin{aligned}
  a_k(u^k, v) + b_k(v, \vrho^k) &= l_k(v), \  \mbox{for all} \ v \in V, \\
  b_k(u^k, \lambda) - c_k(\vrho^k, \lambda) & = \chi_k (\lambda),
  \  \mbox{for all} \ \lambda \in M.
\end{aligned}
\end{equation}
Here, the bilinear forms $a_k\colon V \times V \to \mbb{R}$, $b_k\colon V
\times M \to \mbb{R}$, and $c_k\colon M \times M \to \mbb{R}$ are
defined as  
\begin{align}
  a_k(v, w ) &:= (v, w), \ \mbox{for each} \ v, w \ \in V, \label{eq:ak_defn}\\
  b_k(v,\lambda)  &:= -\Dlt a_{k,k}\frac{\ba}{\veps} (\dvg v, \lambda), \ \mbox{for each}  \ v
  \in V, \ \lambda \in M,  \label{eq:bk_defn}\\
  c_k(\ld, \mu) &:= (\ld, \mu), \ \mbox{for each}  \ \ld, \mu \in
  M, \label{eq:ck_defn} 
\end{align}
and the linear forms $l\colon V\to\mbb{R}$, and $\chi\colon M \to
\mbb{R}$ are defined as
\begin{equation}
\begin{aligned}
  l_k(v) &:= (u^n, v)+\Dlt a_{k,\ell}\frac{\ba}{\veps}(\dvg
              v,\vrho^\ell), \  \mbox{for each} \ v \ \in V,\\
  \chi_k(\lambda) &:= - (\rho^n, \lambda)+ \Dlt a_{k,\ell}\frac{\ba}{\veps}(\dvg
  u^\ell,\ld), \  \mbox{for each} \
  \lambda \ \in M.
\end{aligned}
\end{equation}
Once the existence of $(u^k,\vrho^k)\in V\times M$ is established for
$k=1,2,\ldots,s$,
\eqref{eq:imp_wform_u_n+1}-\eqref{eq:imp_wform_rho_n+1} then uniquely
defines the approximate numerical solution at
$t^{n+1}$. Note that the saddle point problem in
\eqref{eq:imp_wform_u_k}-\eqref{eq:imp_wform_rho_k} is not in the
standard form. We make use of the following result from
\cite{brezzi-fortin-book} to establish the existence and uniqueness of
\eqref{eq:ak_bk}.   

\begin{theorem}[Babu\v{s}ka-Brezzi Inf-sup Theorem]\label{thm:bb}
Let $V$ and $M$ be two Hilbert spaces, and let $a\colon V \times V \to
\mbb{R}$, $b\colon V \times M \to \mbb{R}$, and $c\colon M \times M$ be
three continuous bilinear forms with the following properties.

$a$ is positive semi-definite, i.e.\ $a(v,v)\geq 0$ for all $v\in V$,
and there exists a constant $\alpha>0$ such that
\begin{equation}
  \label{eq:a_bb}
  \inf_{v_0\in V_0}\sup_{w_0\in
    V_0}\frac{a(v_0,w_0)}{\norm{v_0}_V\norm{w_0}_V}\geq \alpha, \
  \mbox{and} \ 
    \inf_{w_0\in V_0}\sup_{v_0\in
    V_0}\frac{a(v_0,w_0)}{\norm{v_0}_V\norm{w_0}_V}\geq \alpha,
\end{equation}
where $V_0 := \{ v_0\in V\colon b(v_0,\ld)=0 \ \mbox{for all} \ \ld
\in M\}$. 

There exists a constant $\beta>0$ such that    
\begin{equation}
  \label{eq:b_bb}
  \sup_{v \in V}\frac{\abs{b(v,\ld)}}{\norm{v}_{V}}
  \geq \beta \inf_{\ld_0 \in M_0}\norm{\ld_0+\ld}_{M}, \ \mbox{for all} \
  \ld\in M,  
\end{equation}
where $M_0 := \{\ld_0 \in M\colon b(v,\ld_0)=0 \ \mbox{for all} \ v
\in V \}$.

$c$ is positive semi-definite and symmetric, i.e.\ $c(\ld,\ld)\geq 0$ for
all $\ld\in M$, and $c(\ld,\mu)=c(\mu,\ld)$ for all $\ld,\mu\in
M$. Further, there exists a constant $\gamma>0$, such that for every
$\ld\in M_0^\bot$, and for every $\epsilon>0$, the solution $\ld_0\in
M_0$ of the equation
\begin{equation}
  \label{eq:c_condn}
  \epsilon (\ld_0,\mu)_M+c(\ld_0,\mu)=-c(\ld,\mu), \ \mbox{for all} \
  \mu\in M,
\end{equation}
is bounded by
\begin{equation}
  \label{eq:c_condn2}
  \gamma\norm{\ld_0}_M\leq \norm{\ld}_M,
\end{equation}
where $(\cdot,\cdot)_M$ and $M_0^\bot$ are, respectively, the
innerproduct in $M$, and the orthogonal complement of $M_0$. 

Finally, let $l \colon V \to \mbb{R}$ and $\chi \colon M \to \mbb{R}$ be two
continuous linear forms. Then the variational problem: find $(u,\vrho)
\in V \times M$, such that   
\begin{equation}
\begin{aligned}
a(u, v) + b(v, \vrho) &= l(v), \ \mbox{for all} \ v \in V, \\
b(u, \ld) - c(\vrho,\ld) &= \chi(\ld), \ \mbox{for all} \  \ld \in M
\end{aligned}
\end{equation}
has one and only one solution. 
\end{theorem}
\begin{proof}[Proof of Theorem~\ref{thm:time_disc_existence}] 
Clearly, the bilinear forms $a_k, b_k$ and $c_k$ are continuous on their
respective domains, and $a_k$ and $c_k$ are symmetric. The structural
condition \eqref{eq:c_condn2} is trivially satisfied for the bilinear
form $c$ defined in \eqref{eq:ck_defn}; see
\cite{brezzi-fortin-book}. Hence, we are left with verifying only the
conditions \eqref{eq:a_bb} and \eqref{eq:b_bb}. 

From the definition of $b_k$ it follows easily that if $v_0\in V$,
then $\dvg v_0=0$, and if $\ld_0\in M_0$, then $\nabla\ld_0=0$. Hence,
for $v_0\in V_0$, we have
$\norm{v_0}_{\Hdiv}=\norm{v_0}_{{\Sint}^d}$. To prove \eqref{eq:a_bb},
let $v_0\in V_0$. Now,
\begin{equation}
  \label{eq:a_infsup}
  \sup_{w_0\in V_0}\frac{a_k(v_0,w_0)}{\norm{w_0}_V} \geq
  \frac{a_k(v_0,v_0)}{\norm{v_0}_V} = \frac{(v_0,v_0)}{\norm{v_0}_V} =
  \norm{v_0}_V. 
\end{equation}
Hence, 
\begin{equation}
  \label{eq:a_infsupfinal}
  \inf_{v_0\in V_0}\sup_{w_0\in V_0}\frac{a_k(v_0,w_0)}{\norm{w_0}_V
    \norm{v_0}_V} \geq 1.
\end{equation}

Next, we proceed to establish the condition \eqref{eq:b_bb}. Let
$\ld_0\in M_0$, and $\ld\in M$. Corresponding to
$\tilde{\ld}:=\ld_0+\ld \in \Sint$, there exists a unique $\tilde{w}
\in H^{1}_{0} (\taud)$ satisfying   
\begin{equation}
  \label{eq:wform_lap}
  (\nabla \tilde{w},\nabla \mu) = (\tilde{\ld}, \mu) \ \mbox{for all}
  \ \mu \in H^{1}_{0} (\taud).   
\end{equation}
Note that $\nabla\tilde{w} \in\Sint^d$ and $\tilde{w}$ satisfies the
elliptic problem $-\Delta\tilde{w}=\tilde{\ld}$ in the sense of
distributions. Hence, $\dvg\nabla\tilde{w}=-\tilde{\ld}\in\Sint$. In
other words, $\nabla\tilde{w} \in \Hdiv$. Therefore, setting
$\mu=\tilde{w}$ in \eqref{eq:wform_lap} yields 
\begin{align}
    \norm{\nabla\tilde{w}}^{2}_{\Sint^d} &= (\tilde{\ld}, \tilde{w}) \nonumber\\
    &\leq\norm{\tilde{\ld}}_{\Sint}\norm{\tilde{w}}_{\Sint} \nonumber\\
    &\leq\norm{\tilde{\ld}}_{\Sint}\norm{\tilde{w}}_{H^{1}(\taud)} \nonumber\\
    &\leq C \norm{\tilde{\ld}}_{\Sint}\abs{\tilde{w}}_{H^1(\taud)} \nonumber\\
    &=C \norm{\tilde{\ld}}_{\Sint}\norm{\nabla\tilde{w}}_{\Sint^d}. \label{eq:gradw2}
\end{align}
Here, we have used the Poincar{\'e} inequality in the last but one
step. Therefore, we have
\begin{equation}
  \label{eq:gradw_norm}
  \norm{\nabla\tilde{w}}_{\Sint^d}\leq C \norm{\tilde{\ld}}_{\Sint}.
\end{equation}
Further, using $-\Delta\tilde{w}=\tilde{\ld}$, we get
$\norm{\Delta\tilde{w}}_{\Sint}=\norm{\tilde{\ld}}_{\Sint}$. Combining
the above two we obtain
\begin{align}
  \norm{\nabla\tilde{w}}_{\Hdiv}^2&=\norm{\nabla\tilde{w}}^2_{\Sint^d}
                                    +\norm{\Delta\tilde{w}}_{\Sint}^2 \nonumber\\
                                  &\leq (C^2+1)
                                    \norm{\tilde{\ld}}^2. \label{eq:wHdivNorm} 
\end{align}
Since $\ld_0\in M_0$, we must have $b(v,\ld_0)=0$ for all $v\in
V$. Thus, for $0 \neq \tilde{\ld} \in M$,  
\begin{align}
  \sup_{v \in V} \frac{\abs{b_k(v,\ld)}}{\norm{v}_V}&=\sup_{v \in V}
                                                    \frac{\abs{b_k(v,\tilde{\ld})}}{\norm{v}_V}
                                                      \nonumber \\
                                                    &\geq\frac{\abs{b_k(\nabla\tilde{w},\tilde{\ld})}}
                                                      {\norm{\nabla\tilde{w}}_V}
                                                      \nonumber\\
                                                    &= \Dlt \abs{a_{k,k}}\frac{\ba}{\veps}
                                                      \frac{\abs{(\Delta\tilde{w},
                                                      \tilde{\ld})}}{\norm{\nabla\tilde{w}}_V}
                                                      \nonumber\\ 
                                                    & \geq  \Dlt \abs{a_{k,k}}\frac{\ba}{\veps}
                                                      \frac{\norm{\tilde{\ld}}^2_M}{\sqrt{(C^2
                                                      +1)}\norm{\tilde{\ld}}_M}
                                                      \nonumber\\
                                                    &=\frac{\Dlt\abs{a_{k,k}}}{\sqrt{(C^2+1)}}\frac{\ba}{\veps}\norm{\tilde{\ld}}_M\nonumber\\ 
                                                    &=\frac{\Dlt\abs{a_{k,k}}}{\sqrt{(C^2+1)}}\frac{\ba}{\veps}\norm{\ld_0+\ld}_M. \label{eq:b_ineq}
\end{align}
The inf-sup condition \eqref{eq:b_bb} now follows from
\eqref{eq:b_ineq} by taking the infimum over $\ld_0\in M_0$.

Hence, it follows from Theorem~\ref{thm:bb} that the $k^{th}$ stage
\eqref{eq:imp_wform_u_k}-\eqref{eq:imp_wform_rho_k}  of the weak
formulation admits a unique solution $(u^k,\vrho^k)\in V\times M$. 
\end{proof}

\subsection{Asymptotic Preserving Property}
\label{sec:asympt-preserv-prop}

The goal of this section is to prove the AP property of the scheme
\eqref{eq:imp_wform_u_k}-\eqref{eq:imp_wform_rho_n+1}. As mentioned
before, proving the AP property consists of proving the asymptotic
stability and asymptotic consistency.   
\begin{theorem} 
  \label{thm:stab_ap}
  Consider the semi-discrete scheme
  \eqref{eq:imp_wform_u_k}-\eqref{eq:imp_wform_rho_n+1}, and assume
  the conditions of Theorem~\ref{thm:time_disc_existence}.
  \begin{enumerate}
  \item \label{item1:en}
    Then there exists a constant $C_k>0$, such that the numerical
    solution $(u^k,\vrho^k)\in V\times M$ of the $k^{th}$ stage
    \eqref{eq:imp_wform_u_k}-\eqref{eq:imp_wform_rho_k} satisfies the
    energy stability estimate:  
    \begin{equation}
      \label{eq:Ek_stab}
      E^k \leq C_k E^n,
    \end{equation}
    where $C_k$ depends only on the IMEX-RK coefficients, but is
    independent of $\veps$. Consequently, there exists a constant
    $C>0$, such that the numerical solution $(u^{n+1},\vrho^{n+1})$
    satisfies the estimate:   
    \begin{equation}
      \label{eq:En+1_stab}
      E^{n+1} \leq C E^n, 
    \end{equation}
    where $C$ depends only on the matrix $A$ and the vector $\omega$,
    but is independent of $\veps$. In other words, the time
    semi-discrete scheme
    \eqref{eq:imp_wform_u_k}-\eqref{eq:imp_wform_rho_n+1} is stable in
    the $L^2$-norm.    
  \item \label{item2:ap}
    If we assume that the solution $(\vrho^{n}, u^{n})$ at time
    $t^n$ is well-prepared i.e. it admits the decomposition:
    \begin{equation}
      \label{eq:tn_mm}
        \vrho^n=\vrho^n_{(0)}+\veps\vrho_{(1)}^n,  \ 
        u^n=u^n_{(0)}+\veps u_{(1)}^n,
    \end{equation}
    where $(\vrho_{(0)}^n, u_{(0)}^n)\in\mcal{E}$, then, the numerical
    solution $(\vrho^{n+1}, u^{n+1})$ also admits a similar
    decomposition
    \begin{equation}
      \label{eq:tn+_mm}
        \vrho^{n+1}=\vrho^{n+1}_{(0)}+\veps\vrho_{(1)}^{n+1},  \ 
        u^{n+1}=u^{n+1}_{(0)}+\veps u_{(1)}^{n+1},
      \end{equation}
      with $(\vrho_{(0)}^{n+1}, u_{(0)}^{n+1})\in\mcal{E}$, which shows
      consistency with the asymptotic limit as $\veps\to 0$.  
  \end{enumerate}
  Hence, the scheme
  \eqref{eq:imp_wform_u_k}-\eqref{eq:imp_wform_rho_n+1} is asymptotic
  preserving.   
\end{theorem}
\begin{proof}
  We prove only the statement in (\ref{item1:en}), and the statement
  (\ref{item2:ap}) follows as in \cite{ArunSamArxiv}. In order to
  prove (\ref{item1:en}), we proceed as follows. Considering the
  first stage, i.e.\ for $k=1$, we have the fully implicit update:
  \begin{align}
    (u^{1},v)-\Dlt a_{1,1}\frac{\ba}{\veps}(\vrho^{1},\dvg v)
    &=(u^{n},v), \ \mbox{for all} \ v\in V, \label{eq:u1} \\
    -\Dlt a_{1,1}\frac{\ba}{\veps}(\dvg
    u^{1},\ld)-(\vrho^{1},\ld)&=-(\vrho^{n},\ld), \ \mbox{for all}  \
                                \ld\in M.\label{eq:rho1}
  \end{align}
  In the above, taking $v=u^1$ and $\ld=-\vrho^1$, adding the
  resulting equations gives
  \begin{equation}
    \label{eq:E1_step1}
    (u^1,u^1)+(\vrho^1,\vrho^1)=(u^n,u^1)+(\vrho^n,\vrho^1).
  \end{equation}
A successive application of the Cauchy-Schwarz inequality on the right
hand side, and rearranging the terms yields 
  \begin{equation}
    \label{eq:E1_est}
    E^{1} \leq E^n. 
  \end{equation}
  In order to get the estimate for the second stage, i.e.\ for $k = 2$, let
  us consider 
\begin{align}
    (u^{2},v)-\Dlt a_{2,2}\frac{\ba}{\veps}(\vrho^{2},\dvg v)
    &=(u^{n},v)+\Dlt a_{2,1}\frac{\ba}{\veps}(\vrho^{1},\dvg v), \
      \mbox{for all} \ v\in V, \label{eq:u2} \\ 
    -\Dlt a_{2,2}\frac{\ba}{\veps}(\dvg
    u^{2},\ld)-(\vrho^{2},\ld)&=-(\vrho^{n},\ld)+\Dlt
                                a_{2,1}\frac{\ba}{\veps}(\dvg u^{1},\ld), \ \mbox{for all}  \
                                \ld\in M.\label{eq:rho2}
\end{align}
In \eqref{eq:u1}-\eqref{eq:rho1} we set $v=-\frac{a_{21}}{a_{11}}u^2,
\ \ld=\frac{a_{21}}{a_{11}}\vrho^2$, in \eqref{eq:u2}-\eqref{eq:rho2}
we set $v=u^2, \ \ld=-\vrho^2$, and add all the resulting equations to get 
\begin{equation}
  \label{eq:E2_step1}
  (u^2,u^2)+(\vrho^2,\vrho^2)=\left(1-\frac{a_{21}}{a_{11}}\right)\left\{(u^n,u^2)+(\vrho^n,\vrho^2)\right\}+\frac{a_{21}}{a_{11}}\left\{(u^1,u^2)+(\vrho^1,\vrho^2)\right\}.   
\end{equation}
Proceeding similarly as in the case of $k=1$, we can obtain from
\eqref{eq:E2_step1} 
\begin{equation}
  \label{eq:E2_stab}
  E^2\leq C_2 E^n.
\end{equation}
Note that the above procedure is similar to the usual forward
elimination process: in the $k^{th}$ stage, we let $v=u^k$,
$\ld=-\vrho^k$, and eliminate the terms containing $(\vrho^\ell,\dvg
v)$ and $(\dvg u^\ell,\ld)$ for $\ell = 1, 2, \ldots , k-1$ by choosing
the test functions $v$ and $\ld$ appropriately in each of the $k-1$
previous stages. The elimination process is valid under the
Hypothesis~\ref{hyp:H2}. Hence, we have for $k = 1, 2, \ldots, s$ 
\begin{equation}
  E^{k} \leq C_k E^{n}
\end{equation}
for an appropriate constant $C_k>0$ which depends only on the
coefficients of the matrix $A$.

An analogous procedure using the update formulae
\eqref{eq:imp_wform_u_n+1}-\eqref{eq:imp_wform_rho_n+1} finally yields
the stability estimate:
\begin{equation}
  \label{eq:E_n+1_E_n}
  E^{n+1}\leq C E^n,
\end{equation}
where $C>0$ depends only on $A$ and $\omega$, and is independent of $\veps$. 
\end{proof}

\subsection{Asymptotic Accuracy}
\label{sec:asymptotic-accuracy}

The asymptotic accuracy follows under the sufficient condition of
$\mcal{E}$-invariance of the scheme. Since the proof follows similar
lines as that \cite{ArunSamArxiv}, we omit the details here.  
\begin{theorem}
  \label{thm:E_invariance}
  The semi-discrete scheme
  \eqref{eq:imp_wform_u_k}-\eqref{eq:imp_wform_rho_n+1} leaves the
  well-prepared space $\mcal{E}$ invariant, i.e.\ if the data
  $(\vrho^n, u^n)$ at time $t^n$ is in $\mcal{E}$, then
  $(\vrho^{n+1}, u^{n+1})\in \mcal{E}$. As a consequence, the
  semi-discrete scheme is asymptotically accurate. 
\end{theorem}

\section{Analysis of Space-Time Fully-discrete Scheme }
\label{sec:SpaceTimeFullyDiscrete}
In this section we present a space-time fully-discrete scheme obtained
by a finite volume strategy, and its analysis. 
Let the given cartesian spatial domain $\Omega$ be discretised into
rectangular cells of length $\Delta x_1$ and $\Delta x_2$ in $x_1$ and
$x_2$ directions, respectively. For notational conveniences, we define
the spatial differential operators $\mu$ and $\delta$, e.g. 
\begin{equation}	\label{eq:weq_delta_mu_defn}
  \delta_{x_1}\omega_{i,j}:={\omega}_{i+\frac{1}{2},j}-{\omega}_{i-\frac{1}{2},j},
  \ \ \ \
  \mu_{x_1} \omega_{i,j} := \frac{\omega_{i+ \frac{1}{2}, j}+\omega_{i
      - \frac{1}{2}, j}}{2}, 
\end{equation} 
in the $x_1$-direction, with analogous definitions in the
$x_2$-direction.

In order to achieve second order accuracy in space, we follow a
MUSCL strategy. From the piecewise constant cell averages
$U^{n}_{i,j}$ of the unknown function $U$ at time $t^n$, we
reconstruct a piecewise linear interpolant. 
In order to carry out the analysis of the fully-discrete scheme as
done in Section~\ref{sec:TimeSemiDiscrete}, we only consider smooth
solutions, and hence, the discrete slopes in the linear recovery are
approximated using central differences without using any limiters.   
\subsection{Space-time Fully-discrete Scheme}

Applying a finite volume discretisation for the fluxes $F(U)$ and
$G(U)$, we obtain the following fully-discrete scheme corresponding to
\eqref{eq:weq_varrho_k_semi_disc}-\eqref{eq:weq_u_semi_disc_finupd}. 
\begin{definition}\label{fully_disc_schm}
The $k^{th}$ stage of an $s$-stage space-time fully-discrete IMEX-RK
scheme for the wave equation system \eqref{eq:LWESWA} is defined as 
\begin{equation}\label{eq:weq_uk_FD}
  U^{k}_{i,j} =
  U^{n}_{i,j}-\tilde{a}_{k,\ell}\ld_m\delta_{x_m}\mcal{F}_m(U^\ell)_{i,j}-a_{k,l}\ld_m\delta_{x_m}\mcal{G}_m(U^l)_{i,j}, \ \mbox{for each} \ k = 1,2, \ldots , s,
\end{equation}
and the final update is given by 
\begin{equation}
  \label{eq:weq_u+_FD}
  U^{n+1}_{i,j} =
  U^{n}_{i,j}-\tilde{\omega}_k\ld_m\delta_{x_m}\mcal{F}_m(U^k)_{i,j}-\omega_k\ld_m\delta_{x_m}\mcal{G}_m(U^k)_{i,j}.
\end{equation}
Here, the repeated index $m$ takes values in ${1,2}$, and $\lambda_m :=
\frac{\Dlt}{\Delta x_m}$ denote the mesh ratios. 
\end{definition}
In our computations, we use a simple Rusanov-type flux to approximate
the explicit part $F$, and a second-order central flux for the
implicit part $G$, e.g.\ in the $x_1$-direction
\begin{equation}
  \label{eq:FG_flux_defn}
  \begin{aligned}
    \mcal{F}_{1,i+\frac{1}{2},j}(U^\ell)&=\frac{1}{2}\left(F_1\left(U^{\ell,+}_{i+\frac{1}{2},j}\right)+F_1\left(U^{\ell,-}_{i+\frac{1}{2},j}\right)\right)-\frac{\bu_1}{2}\left(U^{\ell,+}_{i+\frac{1}{2},j}-U^{\ell,-}_{i+\frac{1}{2},j}\right),\\
    \mcal{G}_{1,i+\frac{1}{2},j}(U^\ell)&=\frac{1}{2}\left(G_1\left(U^\ell_{i+1,j}\right)+G_1\left(U^\ell_{i,j}\right)\right).
  \end{aligned}
\end{equation}
Here $U_{i+\frac{1}{2},j}^\pm$ denotes the right and left interpolated
states at a right hand vertical edge.

Finally, to maintain the stability, the timestep $\Dlt$ is computed
using the CFL condition
\begin{equation}
  \label{eq:CFL}
  \Dlt\max\left(\frac{\bu_1}{\Delta x_1},\frac{\bu_2}{\Delta x_2}\right)=\nu, 
\end{equation} 
where $\nu<1$ is the given CFL number. Note that the above condition
is the advective CFL condition, and is independent of $\veps$.  

\subsection{Solvability of the Space-time fully discrete scheme}
The aim of this subsection is to establish the existence of a unique
solution to the fully-discrete scheme introduced in
Definition~\ref{fully_disc_schm}; cf.\ also
Theorem~\ref{thm:time_disc_existence}. To this end, we use the theory 
of circulant matrices; see \cite{gray} for more details. In order to
make the exposition simple, we consider a one-dimensional scheme;
extension to two dimensions is straightforward.  
\begin{theorem}
  \label{thm:fully_disc_exist}
  Suppose a discrete numerical approximation $(\vrho_i^n,u_i^n)$ at
  time $t^n$, and a timestep $\Dlt$ are given. Then, under
  Hypothesis~\ref{hyp:H2}, each of the intermediate stages
  \eqref{eq:weq_uk_FD} admit a unique solution $(\vrho_i^k,u_i^k)$ for
  $k=1,2,\ldots,s$. As a consequence, the update step
  \eqref{eq:weq_u+_FD} admits a unique solution
  $(\vrho_i^{n+1},u_i^{n+1})$ at time $t^{n+1} = t^n + \Dlt$.  
\end{theorem}
\begin{proof}
  The proof uses induction on $k$, the number of stages. Note that for any
  $k=1, 2,\ldots, s$, we can rewrite the time semi-discrete scheme
  \eqref{eq:weq_varrho_k_semi_disc}-\eqref{eq:weq_u_k_semi_disc} as
  \begin{equation}\label{eq:weq_FD_rk_uk}
  \begin{aligned}
    \vrho^{k} &= \hat{\vrho}^k - \Dlt \frac{\ba}{\veps} a_{k,k} \D_{x_1} u^k ,
    \\
    u^{k} &= \hat{u}^k - \Dlt \frac{\ba}{\veps} a_{k,k} \D_{x_1} \vrho^k,
  \end{aligned}
\end{equation}
where we have denoted the explicit terms 
\begin{equation}
\begin{aligned}
  \hat{\vrho}^k &= \vrho^n - \Dlt \frac{\ba}{\veps} a_{k,\ell} \D_{x_1} u^\ell ,
\\
\hat{u}^k &= u^n - \Dlt \frac{\ba}{\veps} a_{k,\ell} \D_{x_1} \vrho^\ell.
\end{aligned}
\end{equation}
In \eqref{eq:weq_FD_rk_uk}, when $k=1$, we have the fully implicit
first stage 
\begin{equation}\label{eq:weq_vrho_u_k1}
\begin{aligned}
\vrho_{i}^{1} &= \varrho^{n}_{i} - \Dlt \frac{\ba}{\veps} a_{1,1}
\frac{(u^1_{i+1} - u^1_{i-1})}{2 \Delta x_1}, \\
u_{i}^{1} &= u^{n}_{i} - \Dlt \frac{\ba}{\veps} a_{1,1}
\frac{(\vrho^1_{i+1} - \vrho^1_{i-1})}{2 \Delta x_1}, 
\end{aligned}
\end{equation}
for all $i = 1, 2,\ldots, N$, where $N$ denotes the number of mesh
points. Let us denote 
\begin{equation}
\begin{aligned}
  Z^k &= ( \vrho^k_1, \vrho^k_2, \cdots, \vrho^k_{N-1}, \vrho^k_N
  ), \ \  \mbox{for each} \  k = 1, 2,\ldots, s, \\
  V^k &= ( u^k_1, u^k_2, \cdots, u^k_{N-1}, u^k_N
  ), \ \  \mbox{for each} \  k = 1, 2,\ldots, s.
\end{aligned}
\end{equation}
Therefore, \eqref{eq:weq_vrho_u_k1} can be written as 
\begin{equation}\label{eq:weq_FD_1}
  \begin{pmatrix}
    Z^1 \\
    V^1
  \end{pmatrix} \  = 
  \begin{pmatrix}
    Z^n \\
    V^n
  \end{pmatrix} \  - \beta_1 \begin{pmatrix}
    O & P \\
    P & O
  \end{pmatrix} 
  \begin{pmatrix}
    Z^1 \\
    V^1
  \end{pmatrix}.
\end{equation}
where $\beta_1 := \frac{\Dlt}{2 \Delta x} \frac{\ba}{\veps} a_{1,1}$,
$P :=  \circu (0, 1, \cdots, -1)_N$ is an $N\times N$ circulant matrix
\cite{gray}, and $O$ is the $N \times N$ zero matrix. 
The equation \eqref{eq:weq_FD_1} then gives the linear system 
\begin{equation} \label{eq:weq_FD_K_Mform}
A^1(\veps)\begin{pmatrix}
\vrho^1\\
u^1
\end{pmatrix} = 
\begin{pmatrix}
\hat{\vrho}^1 \\
\hat{u}^1
\end{pmatrix},
\end{equation}
where the block-matrix $A^1(\veps)$ is given by
\begin{equation}\label{eq:j_eps_k}
A^1(\veps) : = \begin{pmatrix}
\mds{1} & \beta_1 P \\
\beta_1 P & \mds{1}
\end{pmatrix},
\end{equation}
with $\mds{1}$ being the $N\times N$ identity matrix. Since $\mds{1}$
and $P$ commute, the determinant of $A^1(\veps)$ is given by, see
\cite{bernstein-book},    
\begin{equation}\label{eq:weq_det_jeps}
\det(A^1(\veps)) = \det \left(\mds{1} - \beta_1^2 P^2\right).
\end{equation}
 As a consequence of the Greshgorin's circle theorem \cite{HJ_85}, it
 can be seen that the numerical range of $-\beta_1^2 P^2$ is
 nonnegative and that of $\mds{1}$ is positive; see \cite{HJ_91}. In
 fact, both these matrices are symmetric, and they both have strictly
 positive eigenvalues. Due to the sub-additivity of the numerical range,
 the eigenvalues of the matrix on the right hand side of
 \eqref{eq:weq_det_jeps} are then nonzero. Hence, $A^1(\veps)$ is
 invertible, which in turn confirms the existence and uniqueness of
 $(Z^1, V^1)$.  
 
Now, for each $k = 2, \cdots, s$, we have 
\begin{equation}\label{eq:weq_FD_k}
  \begin{pmatrix}
    Z^k \\
    V^k
  \end{pmatrix} \  = 
  \begin{pmatrix}
    \hat{Z}^k \\
    \hat{V}^k
  \end{pmatrix} \  - \beta_k \begin{pmatrix}
    O & P \\
    P & O
  \end{pmatrix} 
  \begin{pmatrix}
    Z^k \\
    V^k
  \end{pmatrix},
\end{equation}
where $\beta_k := \frac{\Dlt}{2 \Delta x} \frac{\ba}{\veps} a_{k,k}$
for $k=2,\ldots,s$. Note that $(\hat{Z}^k,\hat{V}^k)$ can be written
in terms of $(Z^\ell,V^\ell)$. As in the case of $k=1$, we can now
construct a block matrix $A^k(\veps)$ with $\beta_1$ replaced by
$\beta_k$ in \eqref{eq:j_eps_k} which can be shown to be
invertible. Hence, by induction, we prove the existence and uniqueness
of the solution $(Z^k, V^k)$. As a consequence, the existence and
uniqueness of $(Z^{n+1},V^{n+1})$ follows. 
\end{proof}
\subsection{Asymptotic Preserving Property.} 
\label{sec:asympt-preserv-prop_FD}
We prove the AP property of the fully-discrete scheme by showing
$l^2$-stability uniformly with respect to $\veps$, and its consistency
in the limit $\veps \to 0$. 
\begin{theorem}\label{thm:Asym_stab_FD}
Consider the fully-discrete scheme
\eqref{eq:weq_uk_FD}-\eqref{eq:FG_flux_defn}, and assume the
conditions of Theorem~\ref{thm:fully_disc_exist}
\begin{enumerate}
  \item \label{item1:en_FD}
  Then, there exists a constant $C_k> 0$ such that the numerical
  solution $(u^k_i, \vrho^k_i, )$ of the $k^{th}$ stage
  \eqref{eq:weq_uk_FD}  satisfies the energy stability estimate:
  \begin{equation}
    \label{eq:fully_disc_Ek}
    E^k \leq C_k E^n,
  \end{equation}
  where the constant $C_k>0$ is independent of $\veps$ and depends
  only on the IMEX-RK coefficients. Consequently, there exists a
  constant $C>0$ such that the numerical solution $(u^{n+1}_i, \vrho^{n+1}_i,
  )$ satisfies the estimate 
  \begin{equation}
    \label{eq:flly_disc_En+1}
    E^{n+1}\leq C E^n,
  \end{equation}
  where $C>0$ is independent of $\veps$, and depends only on the matrix
  $A$ and the vector $\omega$. In other words, the time
    fully-discrete scheme
    \eqref{eq:weq_uk_FD}-\eqref{eq:FG_flux_defn} is stable in
    the $l_2$-norm.   
 \item \label{item2:ap_FD}
If we assume that the solution $(\vrho^n_i, u^n_i)$ at time $t^n$ is
well-prepared i.e.\ it admits the decomposition:
\begin{equation}\label{eq:tn_mm_FD}
  \vrho^{n}_{i} = \vrho^{n}_{(0),i} + \veps \vrho^{n}_{(1),i} , \  
  u^{n}_{i} = u^{n}_{(0),i} + \veps
                u^{n}_{(1),i}, \ \mbox{ for
    all} \ i = 1, \ldots, N,
\end{equation}
where $\frac{\delta_{x_1}\mu_{x_1}}{\Delta x_1} \vrho^{n}_{(0),i} = 0$
and $ \frac{\delta_{x_1}\mu_{x_1}}{\Delta x_1} u^{n}_{(0),i} = 0$, or
in other words $(\vrho^{n}_{(0),i}, u^{n}_{(0),i})$ lives in 
$\mcal{E}$. Here, $\frac{\delta_{x_1} \mu_{x_1}}{\Delta x_1 }$ is the
  discrete derivative introduced by the implicit terms, i.e.\ by
  replacing the derivatives by central differences. Then then, the
  numerical solution $(\vrho^{n+1}_i, u^{n+1}_i)$ also admits the same
  decomposition   
\begin{equation}\label{eq:tn+_mm_FD}
  \vrho^{n+1}_{i} = \vrho^{n+1}_{(0),i} + \veps \vrho^{n+1}_{(1),i} ,\  
  u^{n+1}_{i} = u^{n+1}_{(0),i} + \veps u^{n+1}_{(1),i},  \ \mbox{ for
    all} \ i = 1, \ldots, N,
\end{equation}
i.e. the numerical solution is well-prepared and
$(\vrho^{n+1}_{(0),i}, u^{n+1}_{(0),i})$ lives in $\mcal{E}$, which
shows consistency with the asymptotic limit as $\veps \to 0$.
 \end{enumerate}
Hence, the scheme \eqref{eq:weq_uk_FD}-\eqref{eq:FG_flux_defn} is
asymptotic preserving.
\end{theorem}
\begin{proof}
We prove only the statement in (\ref{item1:en_FD}), and the statement
  (\ref{item2:ap_FD}) follows as in \cite{ArunSamArxiv}.
  The proof of (\ref{item1:en_FD}) uses induction on $k$. For $k = 1$ we have from
  \eqref{eq:weq_FD_1}  
  \begin{equation} \label{eq:z1_v1_FD_m}
    \begin{pmatrix}
      Z^1 \\
      V^1
    \end{pmatrix} =
    (A^1(\veps))^{-1} \begin{pmatrix}
      Z^n \\
      V^n
    \end{pmatrix},
  \end{equation}
  where $A^1(\veps)$ is given by \eqref{eq:j_eps_k}. It has to be
  noted that any circulant matrix $M$ can be diagonalised as, see
  \cite{gray},   
  \begin{equation}
    \Lambda_M : =F_N^* M F_N, 
  \end{equation}
  where $*$ denotes the conjugate transpose, and the matrix $F_N$ is a
  unique unitary matrix consisting of eigenvectors of a circulant
  matrix of size $N$. Hence, $F_N$ is independent of the entries
  of $M$, and it is completely determined by the size $N$ of the
  matrix. The diagonalisation of the matrix $A^1(\veps)$ is given by   
  \begin{equation}
    \begin{aligned}
      \Lambda^1(\veps) &= \diag(F_N, F_N) A^1(\veps) \diag(F_N^*,
      F_N^*) \\
      & = \begin{pmatrix}
        \mds{1} & \beta_{1} \Lambda_P \\
        \beta_{1} \Lambda_P & \mds{1}
      \end{pmatrix},
    \end{aligned}
  \end{equation}
  where $\Lambda_P$ is the diagonal matrix consisting of the
  eigenvalues of the matrix $P$, cf.\ also proof of
  Theorem~\ref{thm:fully_disc_exist}. For any matrix norm,
  $\norm{(A^1(\veps))^{-1}}$ satisfies   
  \begin{equation} \label{eq:j_eps_inv_norm}
    \norm{(A^1(\veps))^{-1}} \leq \norm{\mbox{diag}(F_N, F_N)}
      \norm{\mbox{diag}(F_N^*, F_N^*)} \norm{(\Lambda^1(\veps))^{-1}}.
  \end{equation}
  The above inequality \eqref{eq:j_eps_inv_norm} implies that the
  dependence of the norm $\norm{(A^1(\veps))^{-1}}$ on $\veps$ is only
  through $\norm{(\Lambda(\veps))^{-1}}$. By Proposition 2.8.7 in
  \cite{bernstein-book}, the inverse of $\Lambda^1(\veps)$ is given by 
  \begin{equation}
    (\Lambda^1(\veps))^{-1} =
    \begin{pmatrix}
      (\mds{1} - \beta_1^2 \Lambda_P^2)^{-1} & - \beta_1 \Lambda_P
      (\mds{1} - \beta^2 \Lambda_P^2 )^{-1} \\ 
      \beta_1\Lambda_P (\mds{1} - \beta_1^2 \Lambda_P^2 )^{-1} &
      (\mds{1} - \beta_1^2 \Lambda_P^2)^{-1}  
    \end{pmatrix}.
  \end{equation}
  It can be seen that each block in the above matrix is bounded
  uniformly with respect to $\veps$, and hence
  $\norm{(\Lambda^1(\veps))^{-1}}$ also. Therefore,    
  \begin{equation}
    \norm{(A^1(\veps))^{-1}} \leq C_1,
  \end{equation}
  where $C_1>0$ is a constant independent of $\veps$. As a result, from
  \eqref{eq:z1_v1_FD_m}, we have the estimate 
  \begin{equation}
    E^1 \leq C_1 E^n,
  \end{equation}
  where $E:=\norm{(Z, V)}_{l_2}$ is the energy of the fully-discrete
  solution. For $k = 2$, the solution $(Z^2, V^2)$ is given by  
  \begin{equation}\label{eq:z2_v2_FD}
    \begin{aligned}
      \begin{pmatrix}
        Z^2\\
        V^2
      \end{pmatrix} &= 
      \begin{pmatrix}
        Z^n \\
        V^n
      \end{pmatrix} - \Dlt \frac{\ba}{\veps} \frac{a_{2,1}}{2 \Delta x} D
      \begin{pmatrix}
        Z^1 \\
        V^1
      \end{pmatrix}
      - \Dlt \frac{\ba}{\veps} \frac{a_{2,2}}{2 \Delta x} D
      \begin{pmatrix}
        Z^2 \\
        V^2
      \end{pmatrix}, \\
    \end{aligned}
  \end{equation}
  where $D$ is the $2N \times 2N$ central difference discretisation
  matrix, cf.\ \eqref{eq:weq_FD_1}. From \eqref{eq:z2_v2_FD}, we have
  \begin{equation}\label{eq:z2_v2_FD_m}
    A^2(\veps) \begin{pmatrix}
      Z^2 \\
      V^2
    \end{pmatrix} =  \begin{pmatrix}
      Z^n \\
      V^n
    \end{pmatrix} - M(\veps) \begin{pmatrix}
      Z^1 \\
      V^1
    \end{pmatrix}.
  \end{equation}
  In \eqref{eq:z2_v2_FD_m}, the matrix $M(\veps) = -\Dlt \frac{\ba}{\veps}
  \frac{a_{2,1}}{2 \Delta x} D$. Now using \eqref{eq:z1_v1_FD_m} in
  \eqref{eq:z2_v2_FD_m} yields
  \begin{equation}
    \begin{aligned}
      \begin{pmatrix}
        Z^2 \\
        V^2
      \end{pmatrix} &= (A^2(\veps))^{-1} \left( \begin{pmatrix}
          Z^n \\
          V^n
        \end{pmatrix} - M(\veps) (A^1(\veps))^{-1} \begin{pmatrix}
          Z^n \\
          V^n
        \end{pmatrix} \right), \\
      & =  \left( (A^2(\veps))^{-1} - (A^2(\veps))^{-1} M(\veps) (A^1(\veps))^{-1} \right)  \begin{pmatrix}
        Z^n \\
        V^n
      \end{pmatrix}.
    \end{aligned}
  \end{equation}
  As done in the case of $k=1$, it can be shown that the matrix
  $((A^2(\veps))^{-1} - (A^2(\veps))^{-1} M(\veps) (A^1(\veps))^{-1}
  )$ is uniformly bounded with respect to $\veps$. Hence, we have for
  any matrix norm, there exist a constant $C_2$, independent of $\veps$,
  such that  
  \begin{equation}
    \norm{\left( (A^2(\veps))^{-1} -
        (A^2(\veps))^{-1} M(\veps)^2 (A^1(\veps))^{-1} \right)} \leq C_2,
  \end{equation}
  which leads to the stability estimate
  \begin{equation}
    E^2 \leq C_2 E^n. 
  \end{equation}
  In this fashion, we can show that for each $k = 1, 2,\ldots, s$,
  there exist a constants $C_k>0$, independent of $\veps$, such that 
  \begin{equation}
    E^k \leq C_k E^n,   \  \mbox{for all} \ k = 1, 2,\ldots, s.
  \end{equation}
  Substituting the expressions for $(Z^k, V^k)$ in terms of
  $(Z^n,V^n)$ in the update stage for $(Z^{n+1}, V^{n+1})$, and
  estimating the the $l_2$ norm finally yields the stability bound 
  \begin{equation}
    E^{n+1} \leq C E^n,
  \end{equation}
  where  the constant $C$ is independent of $\veps$. 
\end{proof}
\begin{remark}
  It has to be noted that the above stability analysis presented in
  Theorem~\ref{thm:Asym_stab_FD} does not require any condition on
  $\Dlt$ and $\Delta x$. This is not surprising as we are dealing with
  a fully implicit scheme. Carrying out a similar analysis including
  the advection terms will enforce a CFL-like condition independent of
  $\veps$. In \cite{arun_das_sam, ArunSamArxiv}, we have presented the
  results of an analogous study for a first order accurate IMEX-RK
  scheme for the wave equation system with advection using the
  modified equation analysis; see also
  \cite{zakerzadeh,zakerzadeh-noelle} for related studies on the 
  shallow water model. 
\end{remark}
\subsection{Asymptotic Accuracy.}
As in the semi-discrete case, the asymptotic accuracy is a consequence
of the $\mcal{E}$-invariance.
\begin{theorem}
 Suppose that at time $t^n$ the numerical solution $(\vrho^n_{i},
 u^n_{i})$ is in  $\mcal{E}$, i.e.\ $\vrho^n_{i}=\const$
 and $\frac{\delta_{x_1}\mu_{x_1}}{\Delta x_1} u^n_{i}=0$ for all $i$. Then, at time
 $t^{n+1}$, the numerical  approximation $(\vrho^{n+1}_{i},
 u^{n+1}_{i})$ obtained from the  scheme
 \eqref{eq:weq_uk_FD}-\eqref{eq:FG_flux_defn} satisfy      
 \begin{equation}    \label{eq:weq_disc_E_inv1}
	\vrho^{n+1}_{i}=\const, \ \frac{\delta_{x_1}\mu_{x_1}
          }{\Delta x_1} u^{n+1}_{i}=0, \ \mbox{for all} \
        i.
      \end{equation}
In other words, the fully-discrete scheme
\eqref{eq:weq_uk_FD}-\eqref{eq:FG_flux_defn} keeps
the well-prepared space $\mcal{E}$ invariant. 
\end{theorem}
\section{Numerical Results and Their Analysis}
\label{sec:Numerics}
This section is aimed at presenting the results of numerical
computations performed using the proposed scheme. A detailed analysis
of the numerical results is carried out to support and validate the
theoretical findings. The analysis focuses to numerically corroborate
the following four key properties of the proposed scheme: 
\begin{enumerate}[(i)]
\item uniform second-order convergence with respect to $\veps$;
\item uniform stability with respect to $\veps$; 
\item asymptotic consistency;
\item invariance of the well-prepared space $\mcal{E}$, yielding
  asymptotic accuracy. 
\end{enumerate}
We consider four different test cases to establish each of the above
mentioned qualities of the proposed IMEX-RK finite volume
scheme. First, we consider a smooth data to demonstrate the uniform
second-order convergence by computing the experimental order of
convergence (EOC) for different values of $\veps$. Second,
computations are carried out using a two-dimensional moving vortex in
order to testify that the energy dissipation of the scheme is independent
of $\veps$ thereby establishing the uniform stability,
numerically. The third test-case is a two-dimensional well-prepared
data, aimed at demonstrating the asymptotic consistency. Lastly, we
consider a two-dimensional smooth periodic pulse in the well-prepared
space $\mcal{E}$ to demonstrate the asymptotic accuracy and the
asymptotic order of convergence (AOC); see also \cite{ArunSamArxiv}
for related numerical experiments and their results. We have used the
ARS(2,2,2) variant of the IMEX-RK scheme in all the test problems.  
\begin{remark}
  Our numerical computations are carried out using a reformulation of
  the semi-implicit scheme. First, an elliptic equation for the
  density is obtained by eliminating the velocity between the mass and
  the momentum updates. The linear system resulting from the elliptic
  equation is solved using the linear algebra sparse matrix solver
  UMFPACK. Finally, an explicit flux evaluation using the computed
  density in the momentum update yields the updated velocity.  
\end{remark}
\subsection{Experimental Order of Convergence}
\label{sec:problem-1}
We consider the following one-dimensional cosine wave data
\begin{equation}
\begin{aligned}
  \rho (0, x_1) &= 1 + \frac{\veps^2}{1.185} (1 + \cos (2 \pi \veps
  x_1)), \\
  u(0, x_1) & = \veps (1 + \cos (2 \pi \veps x_1)).
\end{aligned}
\end{equation}
The computational domain is $[-\frac{1}{\veps},\frac{1}{\veps} ]$, and
the boundaries are assumed to be periodic. The cosine wave train is
let to  complete three cycles in the domain with an advection velocity
$\bu=1$. The final time $T$ is chosen to be the time taken by the
wave to complete three cycles in domain, i.e.\ $T = 3 \times
\frac{2}{\bar{u} + 1/ \veps}$.  

The simulations are performed for different values of $\veps$ ranging
in $\{10^{0}, 10^{-1}, 10^{-2}, 10^{-3}\}$. As the computational domain
and the final time change with $\veps$, the EOC is obtained with
respect to the mesh size rather than to the number of mesh points. The
CFL number is fixed at $0.45$. The EOC is computed using 
$L^1$ and $L^2$ errors in $\rho$ and $u$ using the exact solution
of the problem as the reference solution. The
Tables~\ref{tab:eocepsm0}-\ref{tab:eocepsm3} clearly show that the
scheme achieves second order convergence uniformly with respect to
$\veps$.  
\begin{table}[htbp]
  \centering
  \begin{tabular}[htbp]{|c|c|c|c|c|c|c|c|c|}
    \hline
    $N$ &$\Delta x_1$ &$L^1$ error in $\rho$ & $L^1$ error in $u$& EOC & $L^2$
                                                               error
                                                               in
                                                               $\rho$
    & $L^2$ error in $u$ &EOC \\  
  \hline	

25 &  0.080000 & 1.533e-02 & 1.818e-02&
                                                                       &
                                                                         1.702e-02
    & 2.018e-02 & \\
\hline
50  & 0.040000 & 3.412e-03 & 4.048e-03 & 2.1677&
                                                                   3.788e-03
    & 4.494e-03 & 2.1676 \\
\hline
100 & 0.020000 & 8.214e-04 & 9.741e-04 &2.0549 & 9.122e-04 &
                                                                      1.081e-03
                         &2.0543\\
\hline
200 & 0.010000 & 2.035e-04 & 2.412e-04 & 2.0133 & 2.260e-04 &
                                                                      2.679e-04
                                                                      &2.0131\\
\hline
  \end{tabular}
  \caption{ $L^1$, $L^2$ errors in $\rho$, $u$, and  EOC for
    Problem~\ref{sec:problem-1} corresponding to $\veps=10^{0}$.}
  \label{tab:eocepsm0}
\end{table}
\begin{table}[htbp]
  \centering
  \begin{tabular}[htbp]{|c|c|c|c|c|c|c|c|c|}
    \hline
    $N$ & $\Delta x_1$&$L^1$ error in $\rho$ & $L^1$ error in $u$& EOC & $L^2$
                                                               error
                                                               in
                                                               $\rho$
    & $L^2$ error in $u$ &EOC \\  
  \hline
50 &  0.400000 & 1.778e-03 & 2.108e-02 & & 1.975e-03 &
                                                                      2.341e-02
   &	\\
\hline
100 & 0.200000 & 4.604e-03 & 5.457e-03 & 1.9498 & 5.113e-04 &
                                                                      6.061e-03
   & 1.9500\\
\hline
200 & 0.100000 & 1.159e-04 &  1.374e-03 & 1.9891 & 1.288e-04 &
                                                                       1.527e-03
                                                                       &1.9889
                                                                         \\
\hline
400 & 0.050000 & 2.907e-05 & 3.446e-04 & 1.9961& 3.229e-05 &
                                                                    3.827e-04
                         & 1.9961\\
\hline
  \end{tabular}
  \caption{ $L^1$, $L^2$ errors in $\rho$, $u$, and  EOC for
    Problem~\ref{sec:problem-1} corresponding to $\veps=10^{-1}$.}
  \label{tab:eocepsm1}
\end{table}
\begin{table}[htbp]
  \centering
  \begin{tabular}[htbp]{|c|c|c|c|c|c|c|c|c|}
    \hline
    $N$ & $\Delta x_1$ &$L^1$ error in $\rho$ & $L^1$ error in $u$& EOC & $L^2$
                                                            error
                                                               in
                                                            $\rho$
    & $L^2$ error in $u$ &EOC \\  
    \hline	
800 & 0.250000 & 6.954e-06 & 8.240e-04 & & 7.724e-06 &
                                                                      9.153e-04
                         & \\
\hline
1600 & 0.125000 & 1.771e-06 & 2.098e-04 & 1.9732 & 1.967e-06 &
                                                                      2.331e-04
                         &1.9732
 \\
\hline
3200 & 0.062500 & 4.443e-07 & 5.265e-05 &1.9950 & 4.935e-07 &
                                                                      5.848e-05
                         & 1.9950\\
\hline
6400 & 0.031250 & 1.113e-07 & 1.319e-05 &1.9961 &1.237e-07 &
                                                                      1.465e-05
                         &1.9961 \\
\hline
  \end{tabular}
  \caption{ $L^1$, $L^2$ errors in $\rho$, $u$, and  EOC for
    Problem~\ref{sec:problem-1} corresponding to $\veps=10^{-2}$.}
  \label{tab:eocepsm2}
\end{table}
\begin{table}[htbp]
  \centering
  \begin{tabular}[htbp]{|c|c|c|c|c|c|c|c|c|}
    \hline
    $N$ & $\Delta x_1$ &$L^1$ error in $\rho$ & $L^1$ error in $u$& EOC & $L^2$
                                                            error
                                                               in
                                                            $\rho$
    & $L^2$ error in $u$ &EOC \\  
    \hline	
3200 & 0.625000 & 3.903e-07 & 4.626e-04 & & 4.336e-07 &
                                                                       5.138e-04
                                                                       &
                                                                         \\
\hline
6400 & 0.312500 & 1.089e-07 & 1.291e-04 &1.8410 & 1.210e-07 &
                                                                       1.434e-04
                                                                       &1.8415
                                                                         \\
\hline
12800 & 0.156250 & 2.783e-08 & 3.298e-05 & 1.9692& 3.091e-08
    & 3.663e-05 & 1.9692\\
\hline
25600 & 0.078125 & 6.991e-09 & 8.284e-06 & 1.9932& 7.765e-09
    & 9.201e-06& 1.9932\\
   \hline
  \end{tabular}
  \caption{ $L^1$, $L^2$ errors in $\rho$, $u$, and  EOC for
    Problem~\ref{sec:problem-1} corresponding to $\veps=10^{-3}$.}
  \label{tab:eocepsm3}
\end{table}

\subsection{Travelling Vortex}
We formulate a travelling vortex problem as follows 
\begin{equation}
\begin{aligned}
\rho(0,x) &= 1.0, \\
u_1 (0, x) & = - K(r) \sin \theta ,\\
u_2 (0, x) & = K(r) \cos \theta.
\end{aligned}
\end{equation}
Here $r = \sqrt{(x_1-0.5)^2+(x_2-0.5)^2}$, $\theta = \tan^{-1} \left(
  \frac{x_2 - 0.5}{x_1 - 0.5} \right)$, and the radial function $K$ is
defined as 
\begin{equation}
K(r) = \begin{cases}
  5 r, & \text{if} \  r  < 0.2, \\
  2 - 5 r, &\text{if} \ 0.2  \leq  r  < 0.4, \\
  0, & \text{otherwise}. 
\end{cases}
\end{equation}
The vortex is set to move in the domain $[0, 4] \times [0, 1]$ by
prescribing an advection velocity $(\bu_1,\bu_2)=(1,0)$.  The CFL
number used is 0.45, and the boundaries are periodic in both the
directions. The computations are carried out for times $T \in \{ 1, 2
,3\}$ with $\veps$ ranging in $\{10^0, 10^{-1}, 10^{-2},
10^{-3}\}$. In Figure~\ref{fig:mach} we provide the Mach number plots
for the entire range of $\veps$ mentioned before and for each time $T$
from the time-range. For reference, we also plot the 
initial Mach number profile. First, it can be observed from the Mach
number plots that the shape of the vortex doesn't deform,
almost. Second, we can note that the shape of the Mach number profile
can be visually seen to be independent of $\veps$. Hence, it can be
concluded that the numerical dissipation stays independent of
$\veps$. This is further confirmed by the the kinetic energy decay
plot in Figure~\ref{fig:ke_vort} in which we plot the kinetic energy
versus time $T\in [0,3]$ for different values $\veps$ mentioned
above. It can be noted that the decay of kinetic energy is 
almost negligible and the energy decay stays independent of $\veps$, as the
plots corresponding to different values of $\veps$ overlap completely . 
\begin{figure}[htbp]
  \centering
  \includegraphics[width=18cm,scale=0.5, trim={0 14cm 0 5cm}, clip]{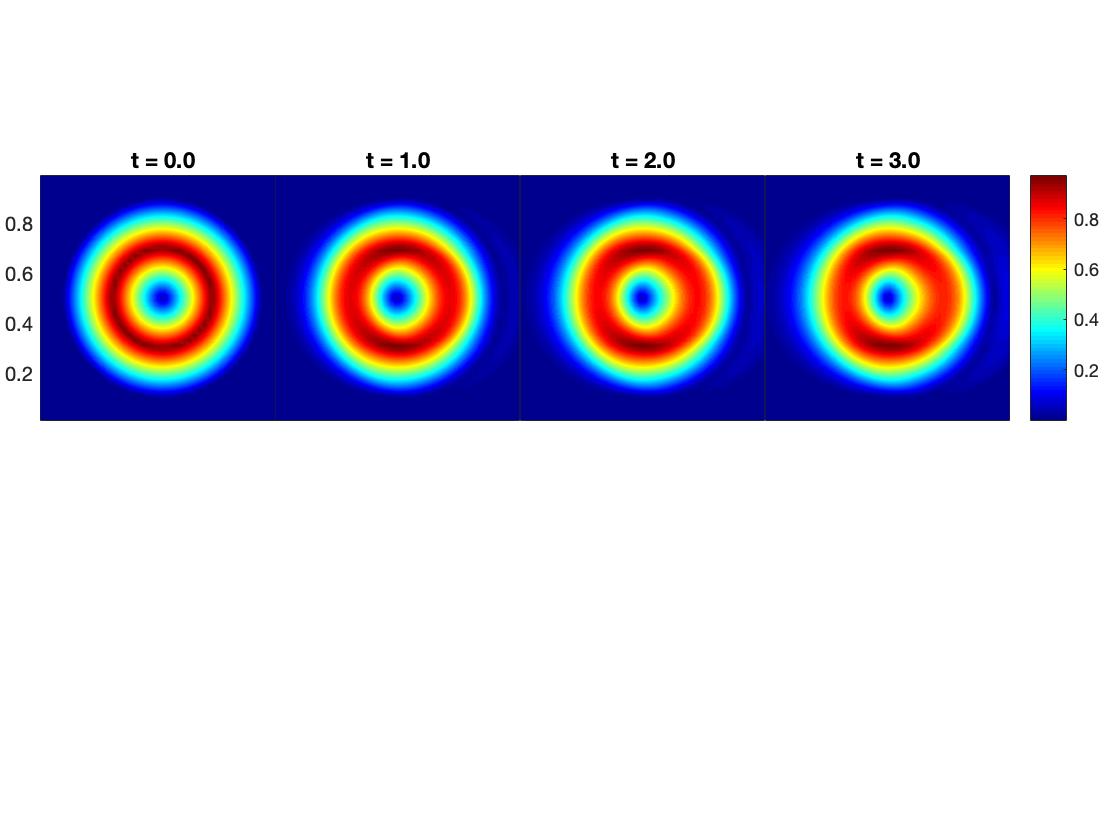}
  \includegraphics[width=18cm,scale=0.5, trim={0 14cm 0 5cm}, clip]{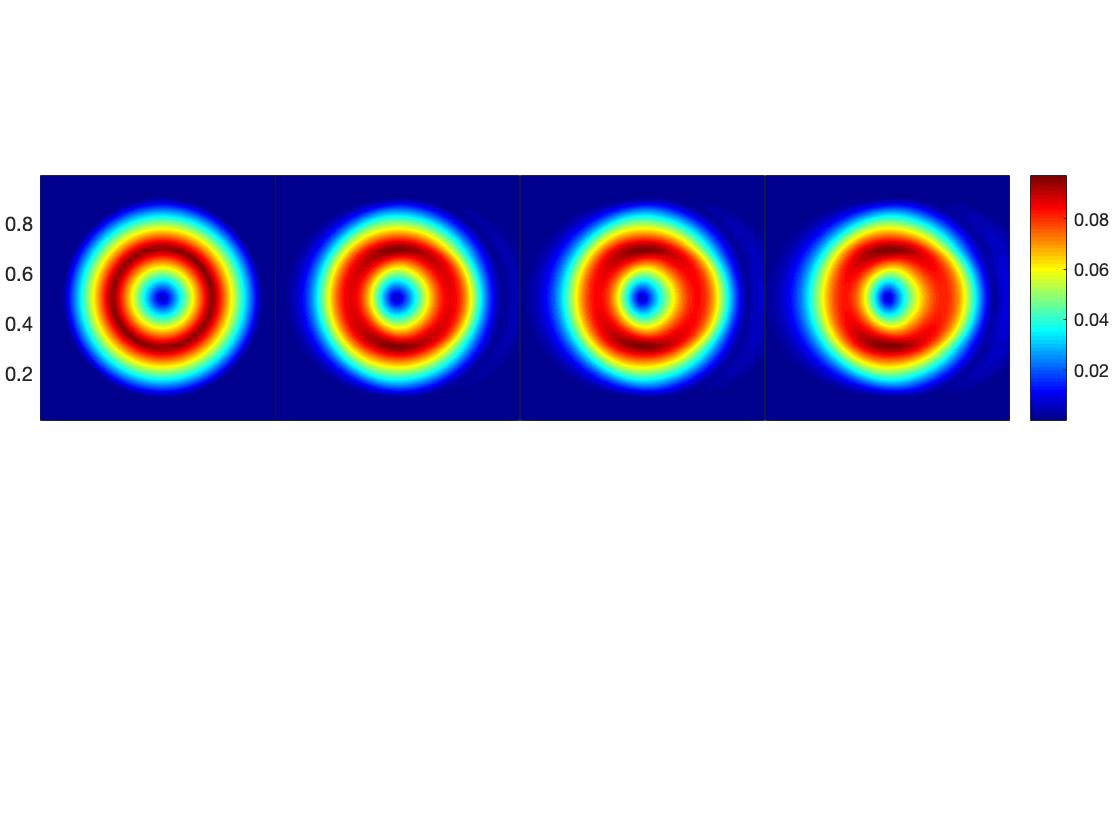}
  \includegraphics[width=18cm,scale=0.5, trim={0 14cm 0 5.2cm}, clip]{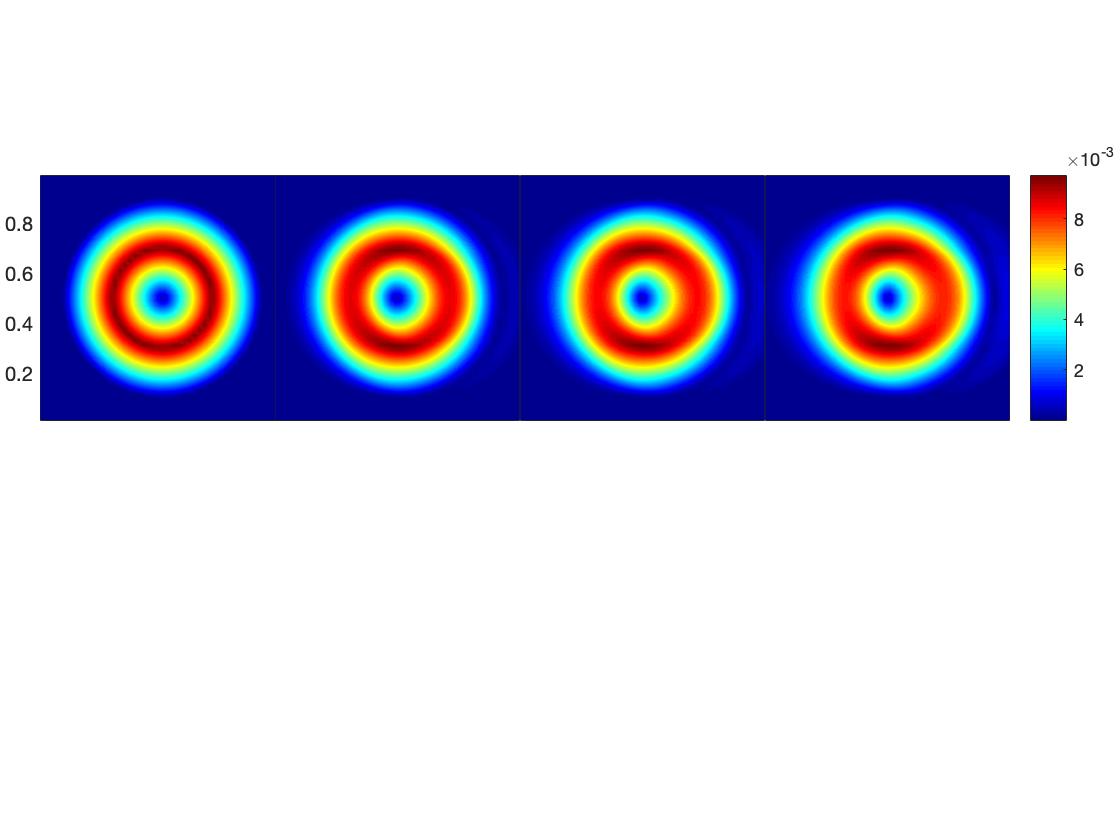}
  \includegraphics[width=18cm,scale=0.5, trim={0 13cm 0 5.2cm},
  clip]{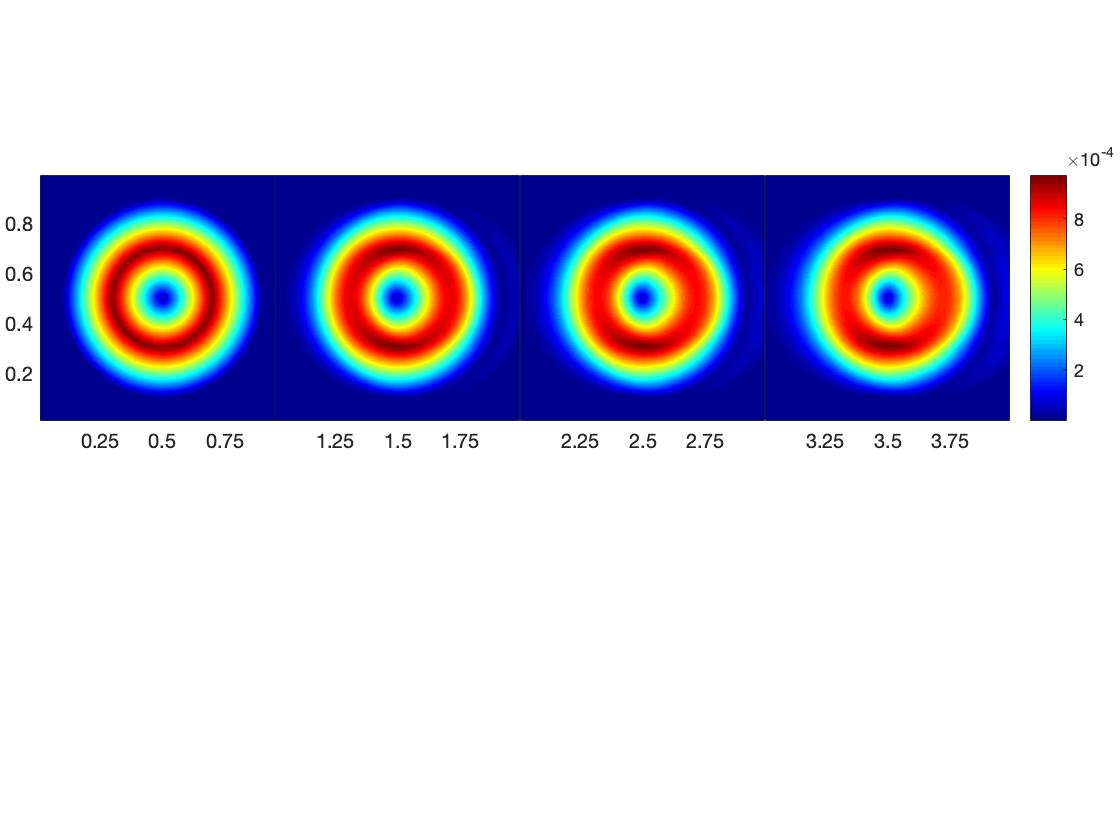}
  \caption{Pseudo-color plots of the Mach numbers for the vortex problem.
    Top: $\veps = 1$, second: $\veps = 10^{-1}$, third: $\veps =
    10^{-2}$, bottom: $\veps = 10^{-3}$, for times $t$ = 0, 1, 2 and 3.}
  \label{fig:mach}
\end{figure}
\begin{figure}[htbp]
\centering
\includegraphics[height=0.28\textheight]{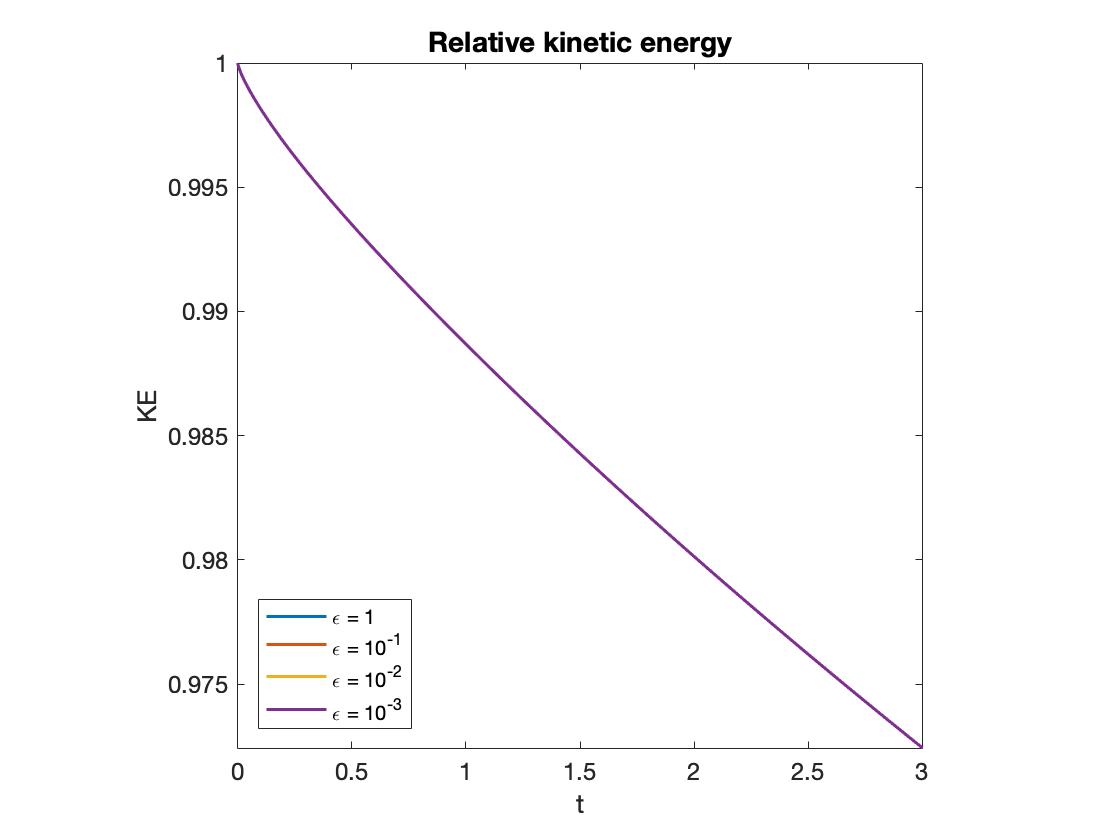}
\caption{Relative kinetic energy from $t = 0$ to $t = 3.0$ for
  different values of $\veps$.
}
\label{fig:ke_vort}
\end{figure}
\subsection{Asymptotic Consistency}
\label{sec:Problem-3}
This test problem is to demonstrate the AP property. Let us consider
the following well-prepared initial data similar to that in
\cite{degond_tang}. 
\begin{equation}
  \begin{aligned}
    \rho(0, x_1, x_2) &= 1 + \veps^2 \sin^2 (2 \pi (x_1 + x_2)), \\     
    u(0, x_1, x_2) &= \sin(2\pi(x_1- x_2)) + \veps \sin(2\pi(x_1 + x_2)), \\
    u(0, x_1, x_2) &=\sin(2\pi(x_1 - x_2)) + \veps \cos(2\pi(x_1+ x_2)).
  \end{aligned}
\end{equation}
We set a very small value of $\veps$, namely $\veps = 10^{-4}$. The
computational domain $[0,1] \times [0, 1]$ is divided into an under
resolved mesh of $40 \times 40$ cells. The boundaries are all taken to
be periodic, and the CFL is 0.45. The linearisation parameters are
$(\bu_1, \bu_2) = (1, 1)$ , $\br = 1$ and $\ba =1$. The final time is
$T = 3$.

In Figure~\ref{fig:grad_div_AP} we plot the density and the divergence
of the velocity at time $T = 0$. Note that the density perturbation is
$O(\veps^2)$ and the divergence perturbation is $O(\veps)$
initially. The corresponding plots obtained using the numerical
solution at time $T = 3$, clearly show that the density is almost
constant, and the velocity divergence is zero. Hence, we culminate that
the numerical solution approximates the incompressible solution
$\rho_{(0)} = \const$, and $\dvg u_{(0)} = 0$, demonstrating the AP
property of the scheme. Further, we show in
Figure~\ref{fig:grad_div_l2_AP} the transient behaviour of
$\norm{\nabla \rho}_{L_2}$ and $\norm{ \dvg u}_{L_2}$ verses time,
from $t = 0$ to $t = 3$. The figure clearly shows that if the initial
data is close to an incompressible data, then the numerical solution
remains close to $\mcal{E}$ for all times.  
\begin{figure}[htbp]
\centering
\includegraphics[height=0.28\textheight]{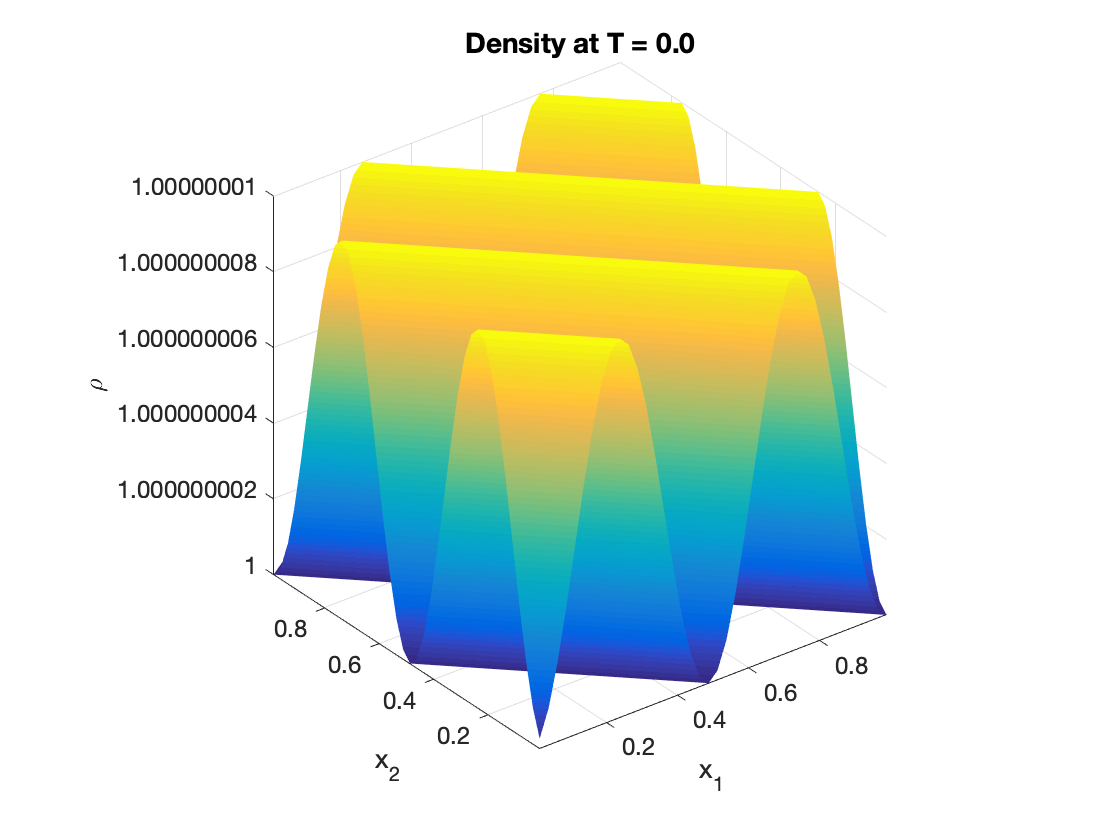}
\includegraphics[height=0.28\textheight]{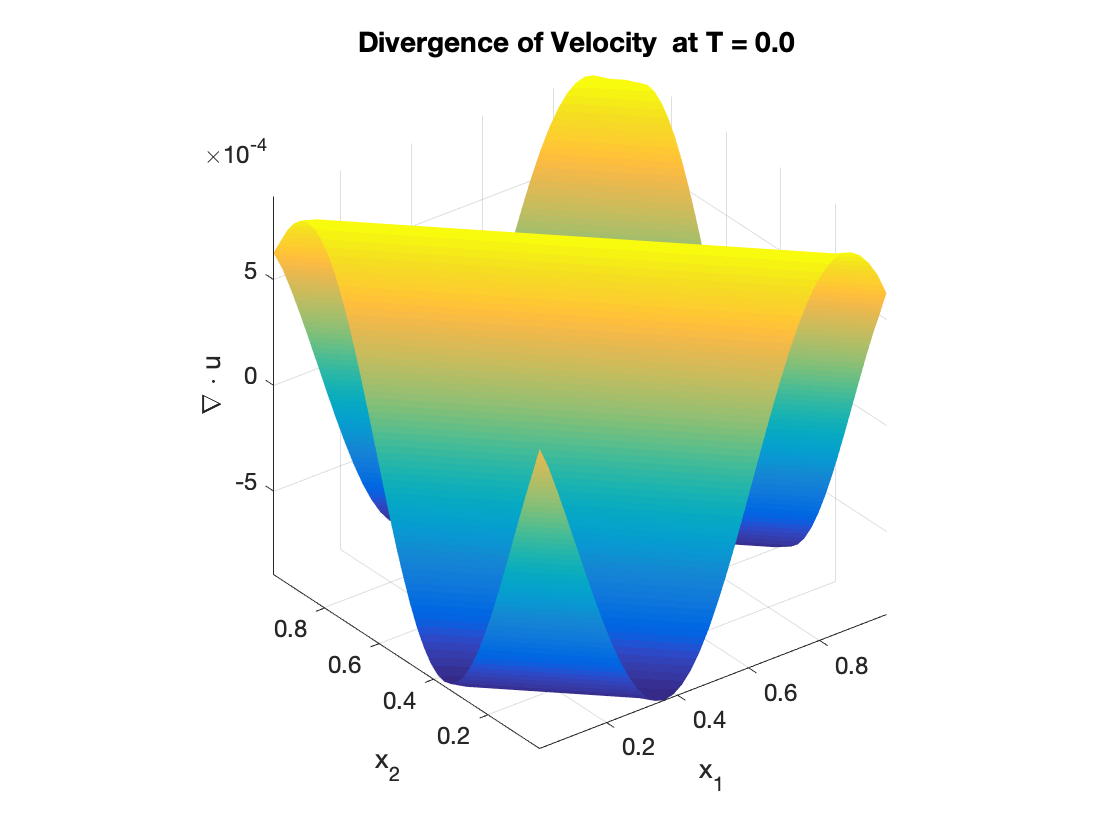}
\includegraphics[height=0.28\textheight]{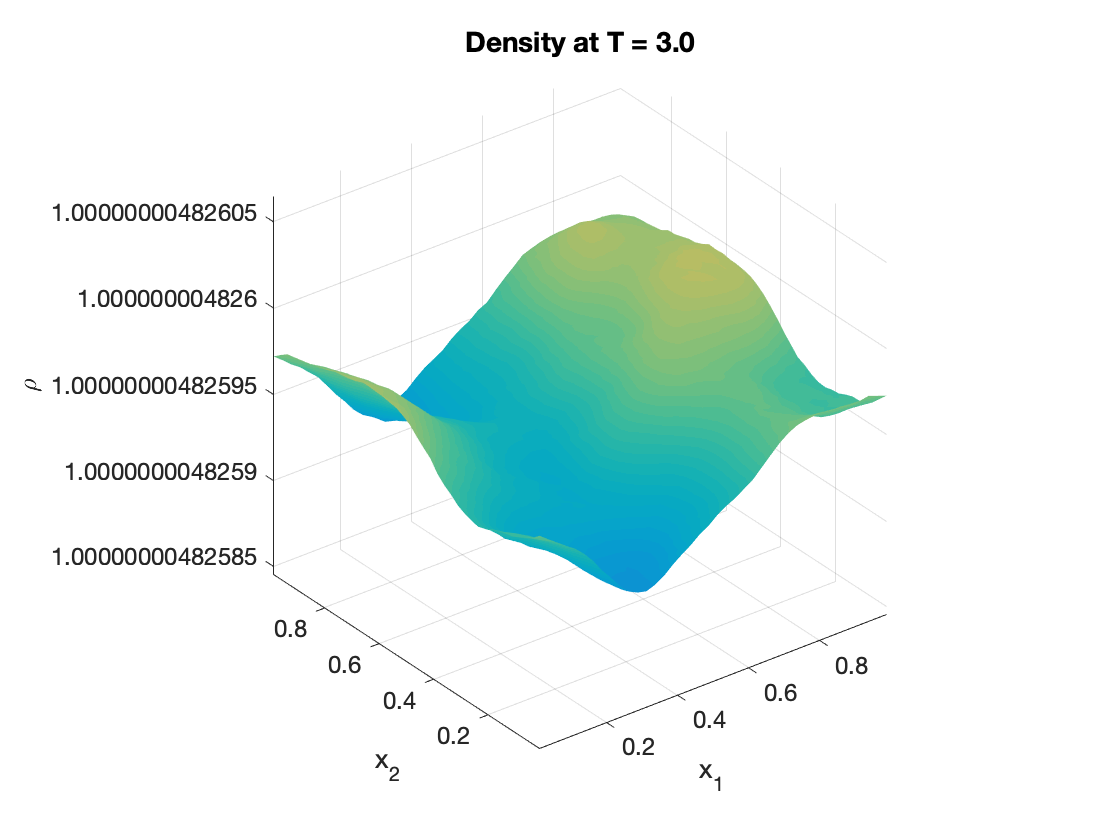}
\includegraphics[height=0.28\textheight]{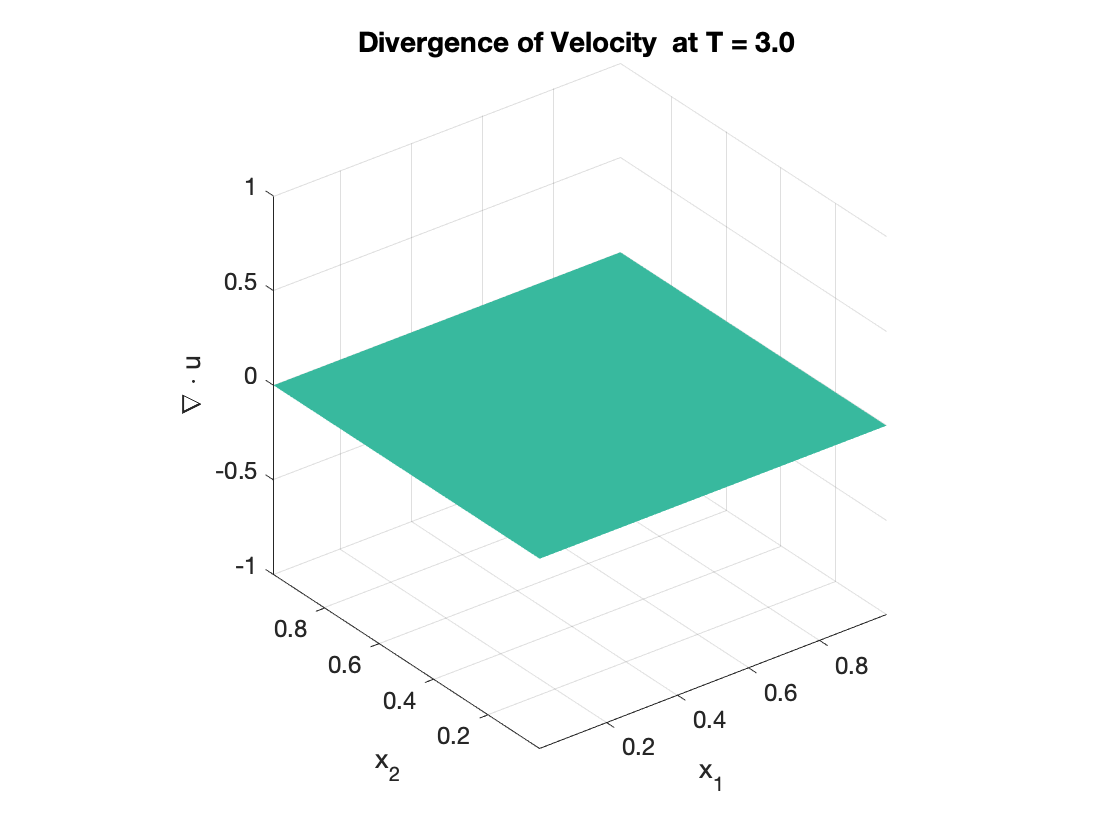}
\caption{Top Left: $\rho$ at $T = 0.0$ . Top Right: $\dvg u$
  at $T = 0.0$. Bottom Left: $\rho$ at $T = 3.0$. Bottom right: $\dvg u$
  at $T = 3.0$. Here, $\veps = 10^{-4}$.}
\label{fig:grad_div_AP}
\end{figure}
\begin{figure}[htbp]
\centering
\includegraphics[height=0.28\textheight]{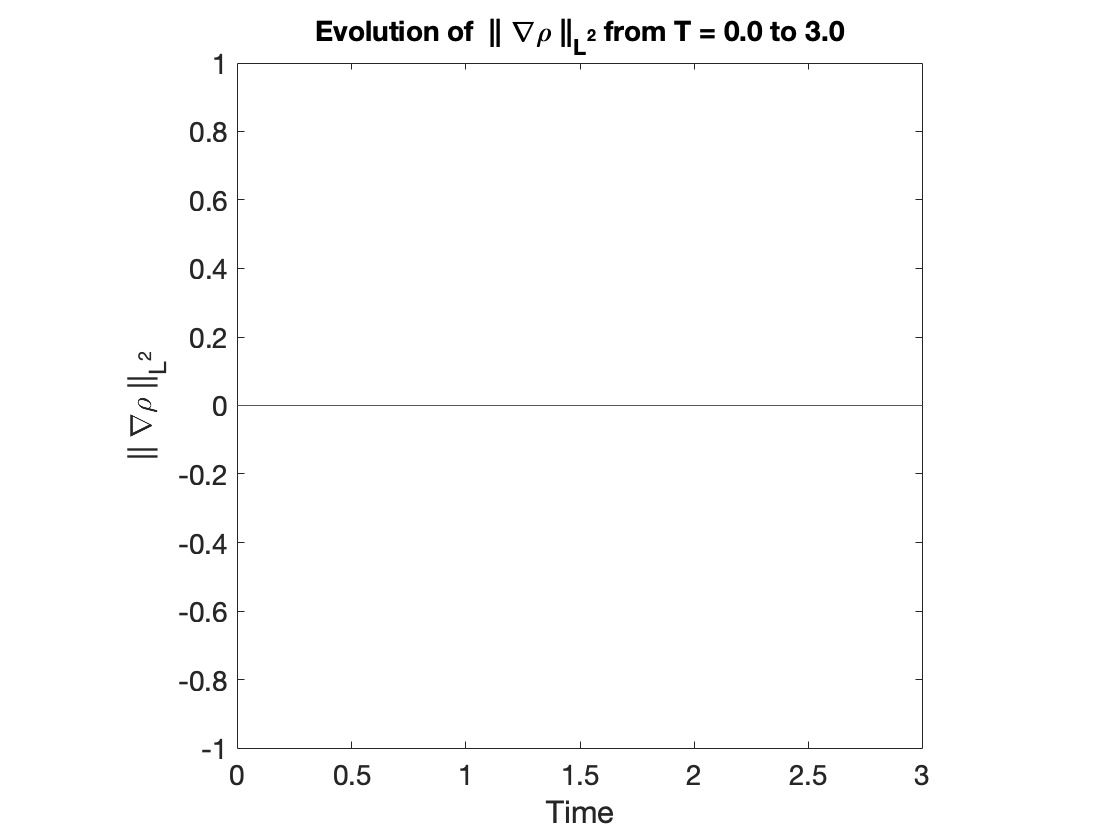}
\includegraphics[height=0.28\textheight]{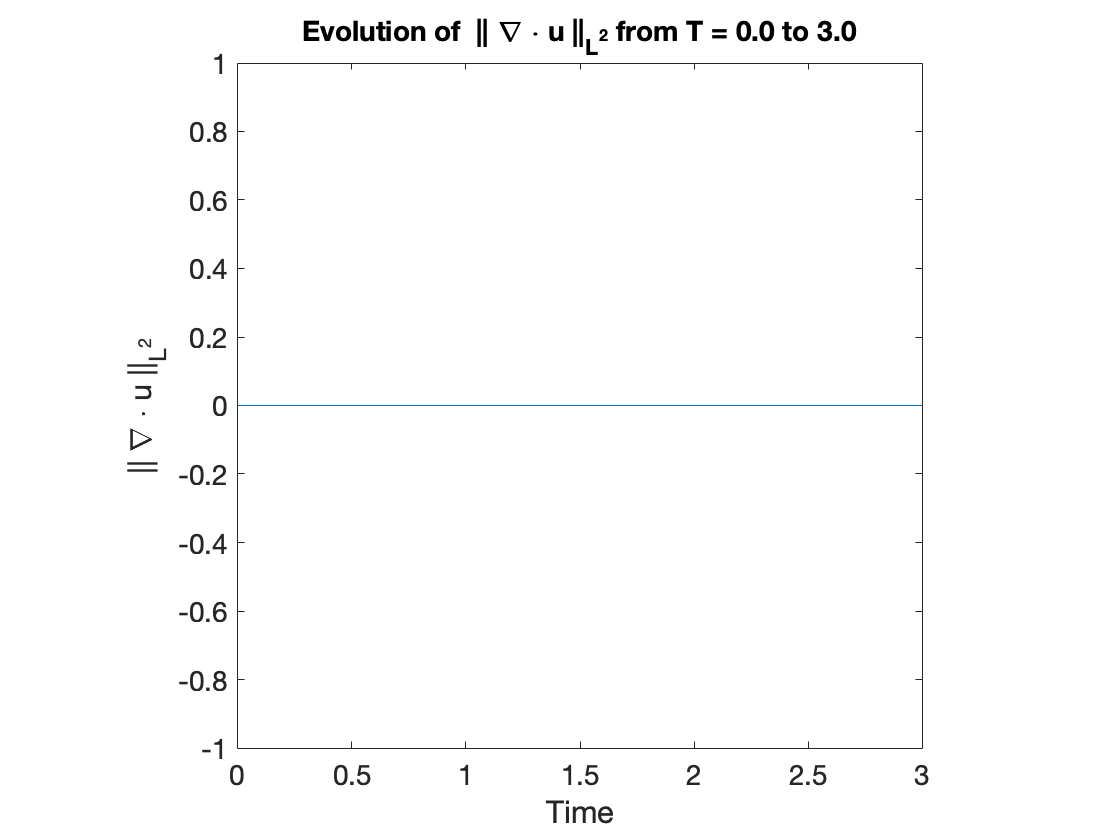}
\caption{Left: $\norm{\nabla \rho}_{L^2}$ from time $T = 0$ to
3. Right: $\norm{ \dvg u}_{L^2}$ from time $T = 0$ to
3. Here, $\veps = 10^{-4}$.}
\label{fig:grad_div_l2_AP}
\end{figure}

\subsection{Asymptotic Order of Convergence and $\mcal{E}$-invariance}
\label{sec:problem-4}
The aim of this experiment is to numerically validate the second order
asymptotic convergence of the numerical solution to the incompressible
limit solution. To this end, we consider the following exact solution of the
incompressible system \eqref{eq:LWESWA_LM} considered in
\cite{schneider}, in which $\vrho_{(0)}(t, x_1, x_2) = 1$, and
\begin{equation}
  \label{eq:schneider}
  \begin{aligned}
    u_{1,(0)}(t, x_1, x_2)&=1-2\cos(2\pi(x_1 - t))\sin(2\pi(x_2-t)), \\
    u_{2,(0)}(t, x_1, x_2)&=1+2\sin(2\pi(x_1 - t))\cos(2\pi(x_2-t)), \\
    \vrho_{(1)}(t, x_1, x_2) &=-\cos(4\pi(x_1-t))-\cos(4\pi(x_2-t)).     
  \end{aligned}
\end{equation}
The linearisation parameters are $\br=1, \ba=1$ and $(\bu_1,
\bu_2)=(1,1)$. The computational domain $\Omega = [0,1] \times [0, 1]$
is successively divided into  $20 \times 20 $, $40 \times 40 $, $80
\times 80$, and $160 \times 160$ mesh cells and the CFL number used is
0.45. The initial data used is obtained by setting
$\vrho(0,x_1,x_2)=\vrho_{(0)}(0,x_1,x_2),
u_1(0,x_1,x_2)=u_{1,(0)}(0,x_1,x_2)$ and
$u_2(0,x_1,x_2)=u_{2,(0)}(0,x_1,x_2)$. 
The boundaries are assumed to be periodic in nature and the final time
for computation is $T=3.0$.   

First, the plot of the $l^1$-norm of the gradient of the density and the
divergence of the velocity at the initial time, and at final time $T = 3$ is
given in Figure~\ref{fig:grad_div}, with an under-resolved $40 \times
40$ mesh for $\veps = 10^{-4}$. The Figure clearly shows that the
IMEX-RK scheme leaves the well-prepared space invariant as the density
is a constant, and the velocity is divergence-free. 

As defined in \cite{ArunSamArxiv}, the EOC computed using
incompressible data $(\vrho_{(0)},u_{1,(0)},u_{2,(0)})$, as the
reference solution is termed as the asymptotic order of convergence
(AOC).The numerical results obtained show that the density $\vrho$ remains
constant exactly at $1$, and hence both the velocity components are
used to measure the AOC. We compute the AOC  for  very small values of
$\veps$, namely for $\veps = 10^{-3}$ and $\veps=10^{-4}$. The AOC
obtained in both the $L^1$ and $L^2$ norms are presented in
Tables~\ref{tab:aocepsm3} and \ref{tab:aocepsm4}. From the tables it 
can easily be seen that as $\veps \to 0$ the numerical solution converges
to the incompressible solution with second order accuracy. This
observation reiterate also the fact that the chosen variant
ARS(2,2,2) is stiffly accurate; see also \cite{AscherRuuthSpiteri}. 

In Figure~\ref{fig:grad_div_l2} we plot the $L^2$ norms of the
gradient of the density and divergence of the velocity over the entire
computational time range $T \in [0,3]$. It clearly shows that the
numerical solution stays in $\mcal{E}$ for all times. Hence, it
substantiate the $\mcal{E}$-invariance with respect to time 

\begin{table}[htbp]
  \centering
  \begin{tabular}[htbp]{|c|c|c|c|c|c|c|c|c|c|}
    \hline
    $N$ & $L^1$ error in $u_1$ & AOC &$L^2$ error in $u_1$& AOC & $L^1$error in $u_2$ & AOC &$L^2$ error in $u_2$ &AOC \\  
    \hline

    20 &  2.670e-01 &	&3.034e-01 &  &2.670e-01&		& 3.034e-01 & \\
      \hline
    40 & 6.931e-02& 1.9461 & 7.749e-02 &1.9692  & 6.931e-02 &1.9461 & 7.749e-02& 1.9692\\
    \hline
    80  & 1.734e-02 &1.9984& 1.930e-02& 2.0054& 1.734e-02 &1.9984 & 1.930e-02&2.0054  \\
    \hline
160& 4.332e-03& 2.0015& 4.814e-03& 2.0034& 4.332e-03 & 2.0015 & 4.814e-03 &2.0034
\\
    \hline
  \end{tabular}
  \caption{ $L^1$, $L^2$ errors in $u_1$, $u_2$, and AOC for
    Problem~\ref{sec:problem-4} corresponding to $\veps=10^{-3}$.}
  \label{tab:aocepsm3}
\end{table}
\begin{table}[htbp]
  \centering
  \begin{tabular}[htbp]{|c|c|c|c|c|c|c|c|c|c|}
    \hline
    $N$ & $L^1$ error in $u_1$ & AOC &$L^2$ error in $u_1$& AOC & $L^1$error in $u_2$ & AOC &$L^2$ error in $u_2$ &AOC \\  
    \hline
20 & 2.670e-01 &	&3.034e-01 & & 2.670e-01 & & 3.034e-01 & \\
     \hline
40 & 6.931e-02 & 1.9461 & 7.749e-02 & 1.9692& 6.931e-02 & 1.9461 & 7.749e-02 &1.9692 \\
    \hline
80 & 1.734e-02 & 1.9984& 1.930e-02 & 2.0054 & 1.734e-02 & 1.9984 & 1.930e-02& 2.0054\\
    \hline
160 & 4.332e-03 & 2.0015& 4.814e-03 & 2.0034 & 4.332e-03 &2.0015 & 4.814e-03 &2.0034\\
    \hline
  \end{tabular}
  \caption{ $L^1$, $L^2$ errors in $u_1$, $u_2$, and AOC for
    Problem~\ref{sec:problem-4} corresponding to $\veps=10^{-4}$.}
  \label{tab:aocepsm4}
\end{table}
\begin{figure}[htbp]
\centering
\includegraphics[height=0.28\textheight]{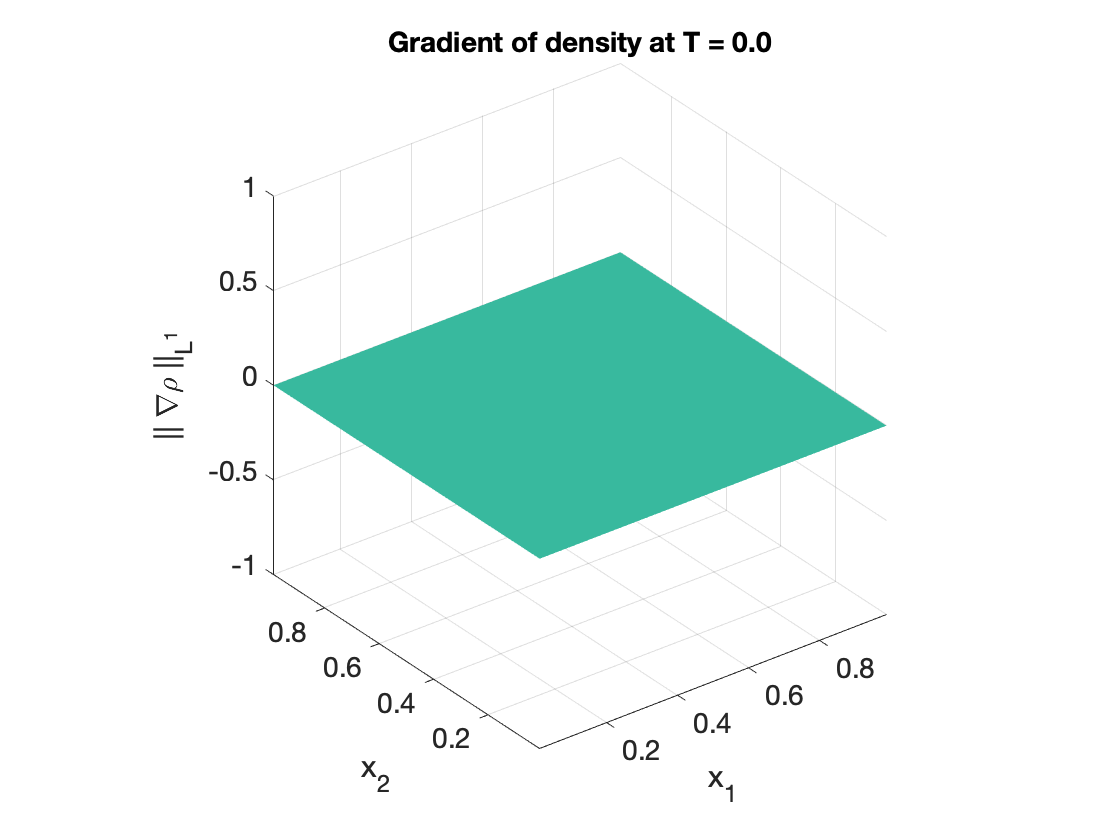}
\includegraphics[height=0.28\textheight]{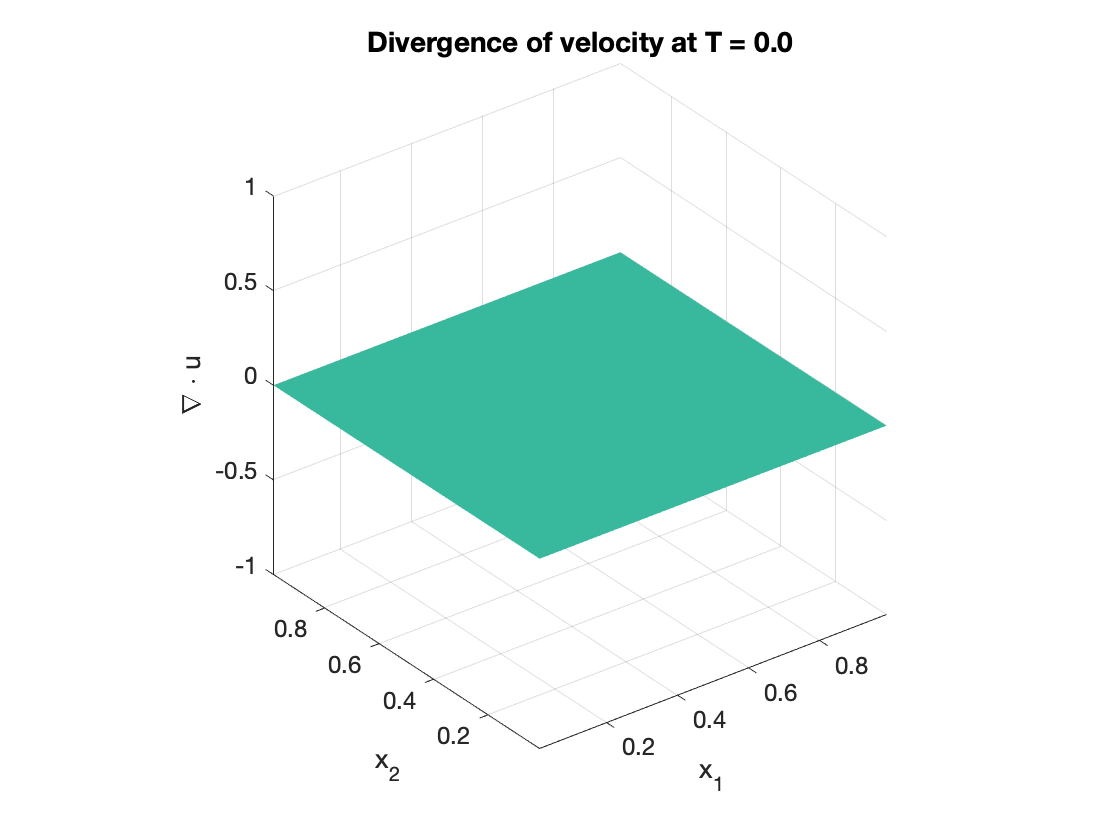}
\includegraphics[height=0.28\textheight]{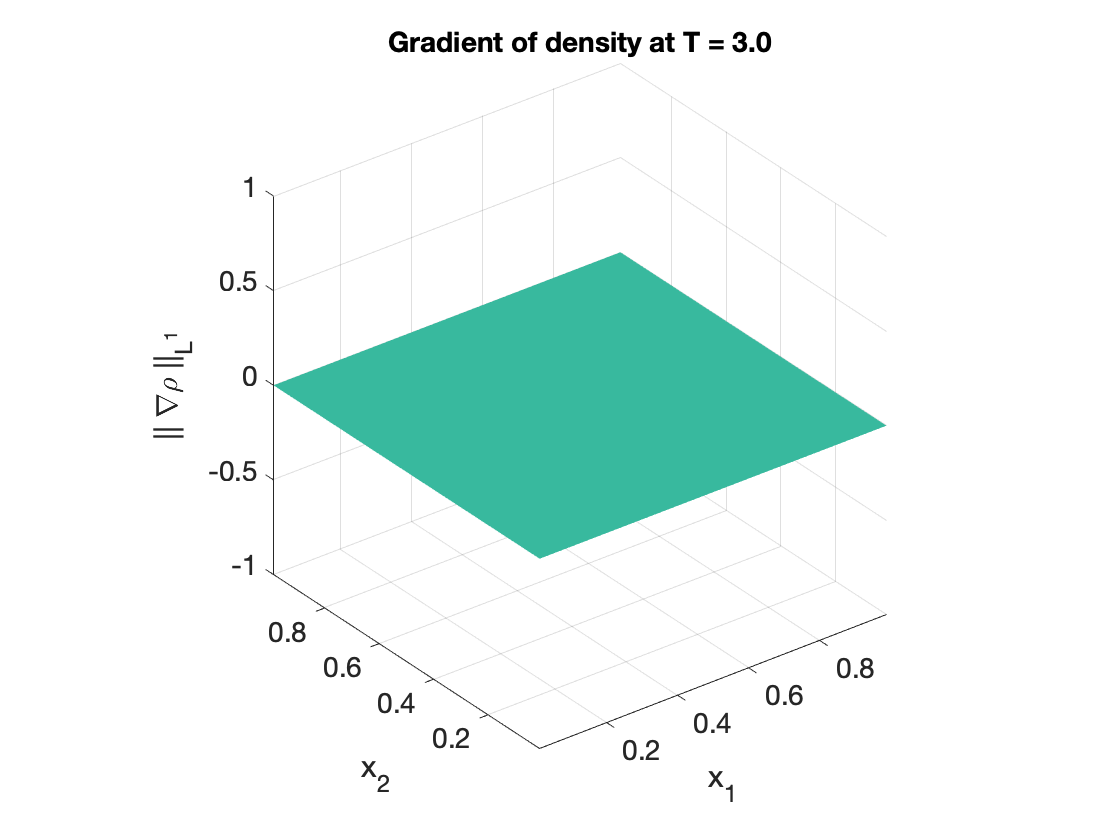}
\includegraphics[height=0.28\textheight]{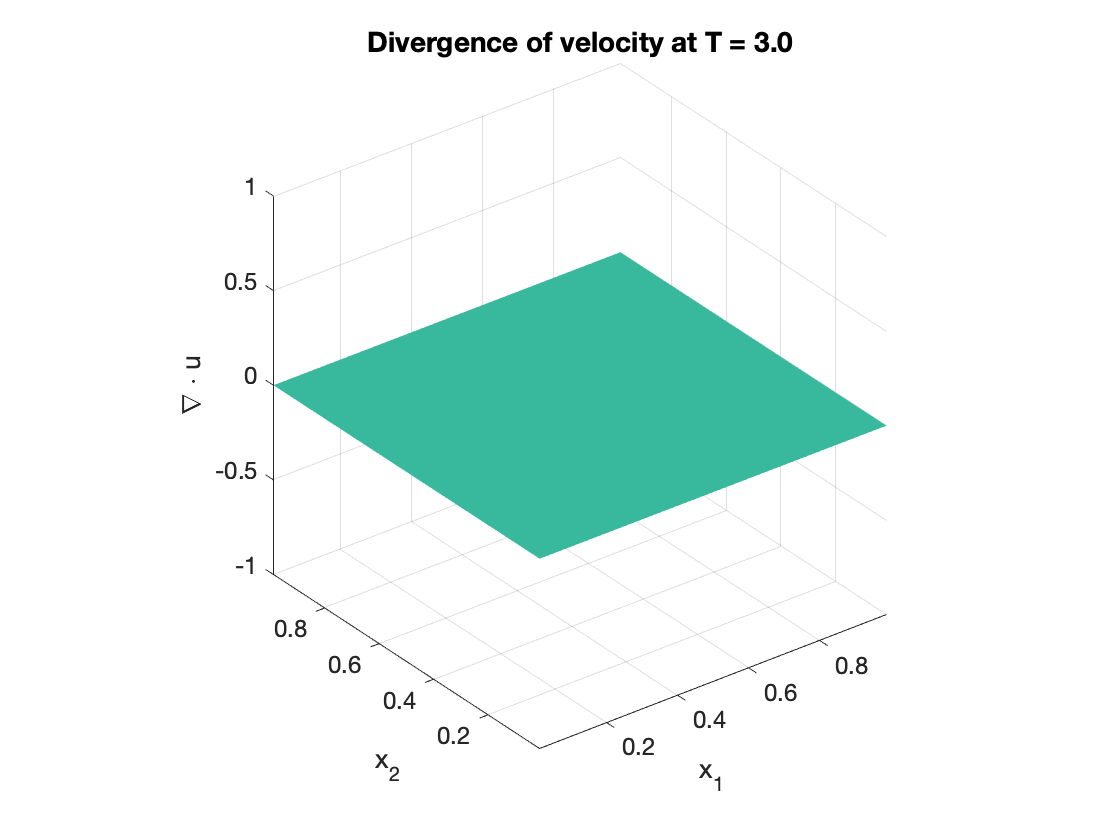}
\caption{Top Left: $\norm{\nabla\rho}_{l^1}$ at $T = 0$.  Top Right:
  $\dvg u$ at $T = 0$. Bottom Left: $\norm{\nabla\rho}_{l^1}$ at $T =
  3$. Bottom right: $\dvg u$ at $T = 3$. Here, $\veps = 10^{-4}$.}
\label{fig:grad_div}
\end{figure}
\begin{figure}[htbp]
\centering
\includegraphics[height=0.28\textheight]{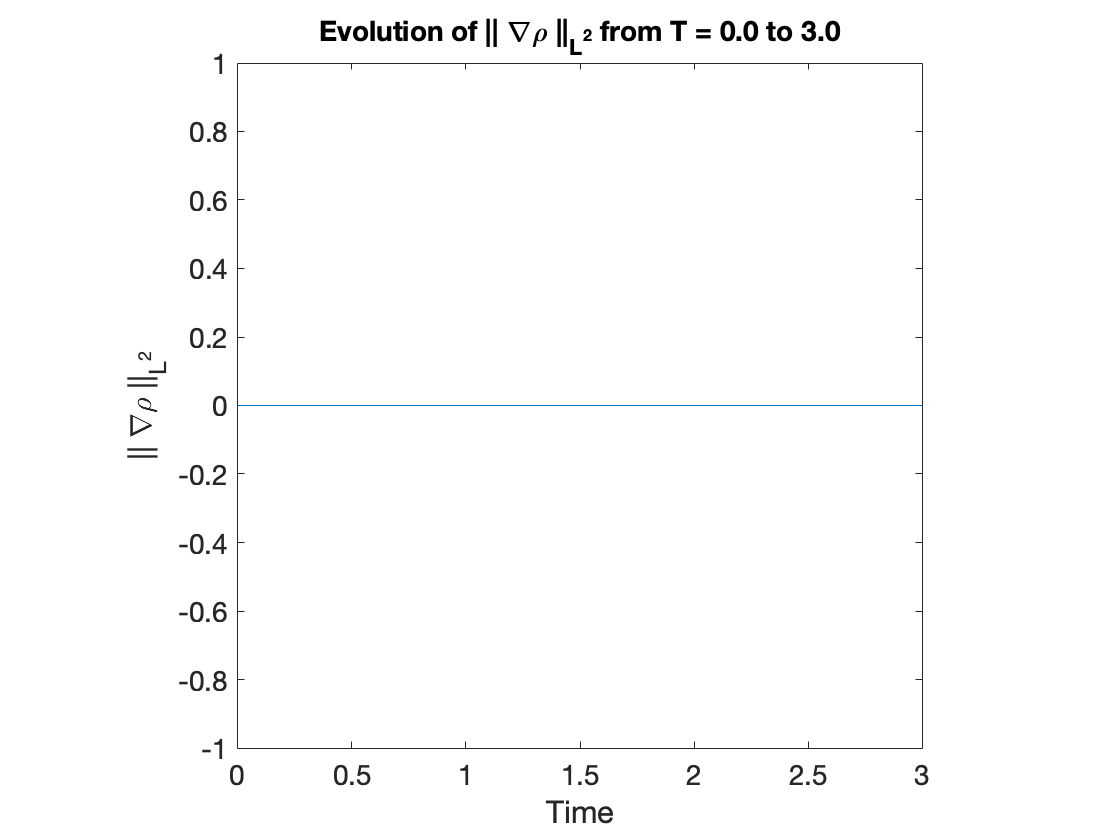}
\includegraphics[height=0.28\textheight]{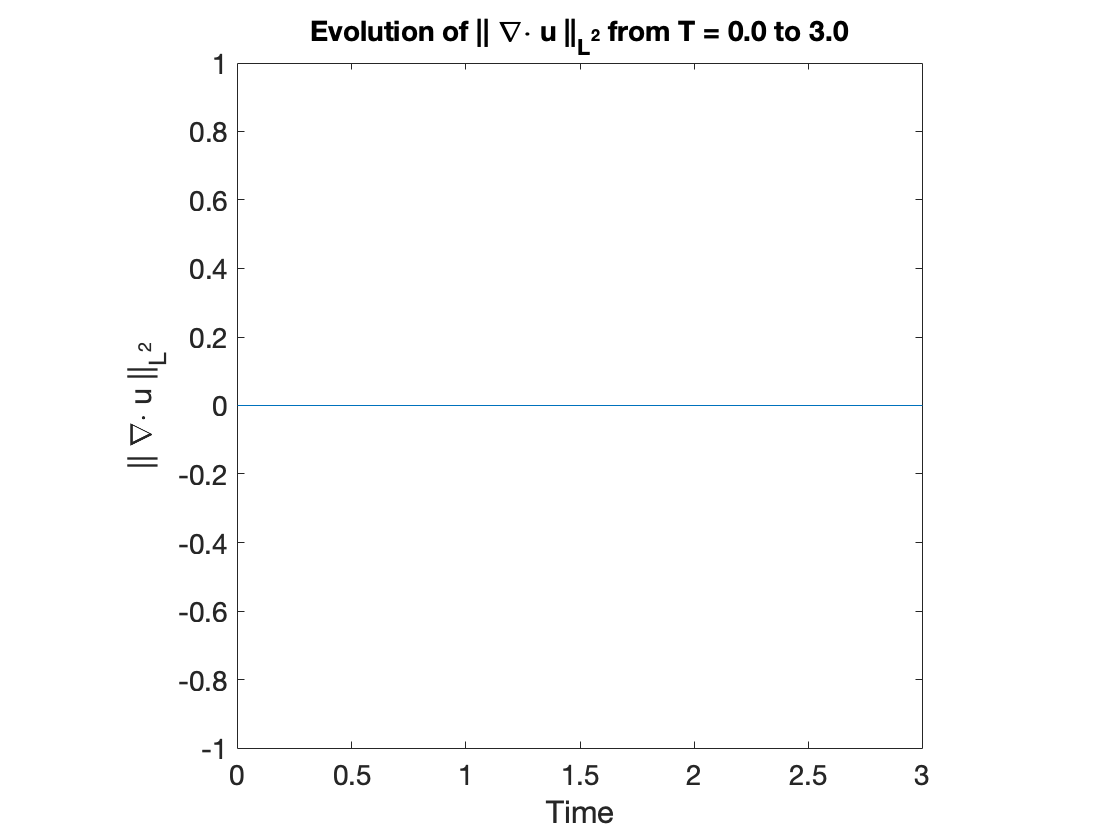}
\caption{Left: $\norm{\nabla \rho}_{L^2}$ from time $T = 0$ to
  3. Right: $\norm{\nabla \cdot u}_{L^2}$ from time $T = 0$ to
  3. Here, $\veps = 10^{-4}$.}
\label{fig:grad_div_l2}
\end{figure}
\section{Concluding Remarks}
\label{sec:Conclusion}
In this paper, we have presented a detailed analysis of an IMEX-RK
finite volume scheme for the linear wave equation system in the zero
Mach number regime. The existence of a unique numerical solution, its
uniform stability with respect to $\veps$, the AP property, and the 
asymptotic accuracy are shown for the time semi-discrete scheme using
saddle point theory of variational problems. Results from the theory
of circulant matrices are used to establish the same features for the
space-time fully-discrete scheme, obtained via a finite volume
discretisation. Extensive numerical studies are carried out to test
the various theoretical concepts discussed. Uniform second-order
convergence is achieved with respect to $\veps$, experimentally. The
dissipation of the scheme is shown to be independent of
$\veps$. The experiments reveal that the scheme is AP, and also
achieves second-order asymptotic convergence, leaving the
well-prepared space $\mcal{E}$ invariant. Hence, the numerical case
studies validate all the theoretical findings.  
\bibliography{references}
\bibliographystyle{abbrv}
\end{document}